\documentclass[reqno]{amsart}

\usepackage{amsmath}
\usepackage{amsfonts}
\usepackage{amsthm}
\usepackage{amssymb}
\usepackage{graphicx}
\usepackage{graphics}
\usepackage{bm}
\usepackage{dsfont}
\usepackage{color}
\usepackage[font=footnotesize]{caption}
\usepackage[all]{xy}
\numberwithin{equation}{section}
\newtheoremstyle{personal}%
{12pt}
{12pt}
{\slshape}
{}
{\bfseries}
{.}
{.5em}
{}
\theoremstyle{personal}%
\newtheorem{thm}{Theorem}[section]

\newtheorem{lem}[thm]{Lemma}
\newtheorem{prop}[thm]{Proposition}

\theoremstyle{definition}

\newcommand{\I}{\mathds{I}}
\newcommand{\V}{\mathds{V}}
\newcommand{\W}{\mathds{W}}
\newcommand{\HH}{\mathds{H}}
\newcommand{\N}{\mathds{N}}
\newcommand{\Z}{\mathds{Z}}
\newcommand{\R}{\mathds{R}}
\newcommand{\Q}{\mathds{Q}}

\newcommand{\PP}{\mathds{P}}
\newcommand{\C}{\mathds{C}}

\newcommand{\zerosection}{0\mbox{-}\mathrm{section}}
\newcommand{\crit}{\mathrm{crit}}
\newcommand{\diff}{\mathrm{d}}
\newcommand{\incl}{\mathrm{incl}}
\newcommand{\ind}{\mathrm{ind}}

\newcommand{\nul}{\mathrm{nul}}

\newcommand{\qq}{\bm{q}}
\newcommand{\pp}{\bm{p}}
\newcommand{\vv}{\bm{v}}
\newcommand{\ww}{\bm{w}}

\newcommand{\injrad}{\mathrm{injrad}}
\newcommand{\Tan}{\mathrm{T}}
\newcommand{\ev}{\mathrm{ev}}
\newcommand{\Ev}{\mathrm{Ev}}

\DeclareRobustCommand{\llongrightarrow}{\relbar\joinrel\relbar\joinrel\rightarrow}
\DeclareMathOperator{\rank}{\mathrm{rank}} 
 
\DeclareMathOperator*{\toup}{\longrightarrow} 
\DeclareMathOperator*{\ttoup}{\llongrightarrow}

\begin{document}

\title{A min-max characterization of Zoll Riemannian metrics}

\author{Marco Mazzucchelli}
\address{Marco Mazzucchelli\newline\indent CNRS, \'Ecole Normale Sup\'erieure de Lyon, UMPA\newline\indent  46 all\'ee d'Italie, 69364 Lyon Cedex 07, France\newline\indent and\newline\indent Mathematical Sciences Research Institute\newline\indent 17 Gauss Way, Berkeley, CA 94720, USA}
\email{marco.mazzucchelli@ens-lyon.fr}

\author{Stefan Suhr}
\address{Stefan Suhr\newline\indent Ruhr-Universit\"at Bochum, Fakult\"at f\"ur Mathematik\newline\indent
Geb\"aude NA 4/33, D-44801 Bochum, Germany}
\email{stefan.suhr@rub.de}

\date{September 23, 2018. \emph{Revised}: October 21, 2018}

\begin{abstract}
We characterize the Zoll Riemannian metrics on a given simply connected spin closed manifold as those Riemannian metrics for which two suitable min-max values in a finite dimensional loop space coincide. We also show that on odd dimensional Riemannian spheres, when certain pairs of min-max values in the loop space coincide, every point lies on a closed geodesic.
\end{abstract}
\maketitle

\section{Introduction}
\label{s:intro}

On a closed manifold of dimension at least 2, a Riemannian metric is called Besse when all of its geodesics are closed. It is  called Zoll when all its unit-speed geodesics are closed with the same minimal period, and simple Zoll when they are also without self-intersections. As usual, by closed geodesic we mean a non-constant periodic orbit of the geodesic flow. Riemannian metrics in these three classes are of great interest in Riemannian geometry, see \cite{Besse:1978pr}.

The only known closed manifolds admitting Zoll Riemannian metrics are the compact rank-one symmetric spaces, that is, $S^n$, $\R \PP^n$, $\C \PP^n$, $\HH \PP^n$, or $\mathrm{Ca}\PP^2$, whose canonical Riemannian metrics are simple Zoll. Actually a result of Bott and Samelson \cite{Bott:1954aa, Samelson:1963aa} implies that any closed manifold admitting a simple Zoll Riemannian metric has the integral cohomology ring of a compact rank-one symmetric space.
Conjecturally, on simply connected closed manifolds $M$, the notions of Besse and Zoll Riemannian metrics are equivalent. This conjecture has been recently established for $M=S^n$ with $n\geq4$ by Radeschi and Wilking \cite{Radeschi:2017dz}, and was earlier established for $M=S^2$ by Gromoll and Grove \cite{Gromoll:1981kl}, who also showed that on $S^2$ the condition of being simple Zoll is equivalent to the other two ones.

The geodesic flow of a Riemannian manifold is a classical autonomous Hamiltonian flow in its tangent bundle. This implies that the closed geodesics parametrized with constant speed are the non-trivial critical points of the energy functional on the space of loops, whereas the closed geodesics with any parametrization are the critical points of the length energy functional. A result claimed by Lusternik \cite{Ljusternik:1966tk} and proved recently by the authors \cite{Mazzucchelli:2017aa} implies that a Riemannian 2-sphere is Zoll if and only if the min-max values of the length functional over three suitable homology classes of the space of unparametrized simple loops coincide. If instead only two among these three values coincide, the Riemannian metric may not be Zoll, but the geodesic dynamics is still rather special: any point of the 2-sphere must lie on a closed geodesic. For $n$-spheres of arbitrary dimension $n\geq2$ with sectional curvature pinched inside $[1/4,1]$, analogous results were proved by Ballmann-Thorbergsson-Ziller \cite{Ballmann:1983fv}. The goal of this paper is to provide further results along this line for more general closed manifolds, and in particular for higher dimensional ones, without any assumption on the curvature.

In order to state our theorems, let us quickly recap the variational theory for the closed geodesic problem. Let $M$ be a closed orientable manifold of dimension $n\geq 2$ admitting a simple Zoll Riemannian metric. The manifold $M$ will be implicitly identified with the submanifold of constant loops in the free loop space $\Lambda M=W^{1,2}(\R/\Z,M)$. For this class of manifolds, we have the explicit cohomology computation
\begin{align}
\label{e:cohomology_loop_space}
H^*(\Lambda M,M)
\cong 
\bigoplus_{m\geq1}
H^{*-m\,i(M)-(m-1)(n-1)}(SM),
\end{align}
where $SM$ denotes the unit tangent bundle of $M$, and $i(M)$ is a suitable positive integer only depending on the integral cohomology ring of $M$, according to an argument due to Radeschi and Wilking \cite[page~942]{Radeschi:2017dz}. Throughout this paper, the singular cohomology $H^*$ and the singular homology $H_*$ will always be intended with $\Z$ coefficients unless we explicitly state otherwise. Since $M$ admits a simple Zoll Riemannian metric, Bott and Samelson's theorem \cite[Theorem~7.23]{Besse:1978pr} implies that it is simply connected and with vanishing Euler characteristic if and only if it is homeomorphic to an odd dimensional sphere $S^n$. In this case, the relative cohomology group $H^{*}(\Lambda M,M)$ has rank at most one in every degree, and for each integer $m\geq1$ we choose two generators
\begin{equation}
\label{e:classes_thm1}
\alpha_m  \in H^{(2m-1)(n-1)}(\Lambda M,M),
\qquad
\beta_m  \in H^{2m(n-1)+1}(\Lambda M,M).
\end{equation}
For each Riemannian metric $g$ on $M$, the associated energy functional is 
\[
E:\Lambda M\to[0,\infty),\qquad E(\gamma)=\int_0^1 \|\dot\gamma(t)\|^2_g\,\diff t.
\]
For each $b>0$, we consider the energy sublevel set $\Lambda M^{<b}:=\{\gamma\in\Lambda M\ |\ E(\gamma)<b\}$, and denote by $\iota_b:(\Lambda M^{<b},M)\hookrightarrow(\Lambda M,M)$ the inclusion. Given a non-trivial cohomology class $\mu\in H^d(\Lambda M,M)$, the associated min-max 
\begin{align*}
c_g(\mu)=c_g(-\mu):=\inf\{b>0\ |\ \iota_b^*\mu\neq 0\}
\end{align*}
is a critical value of $E$, and thus the energy of a closed geodesic.
One can easily verify that, if $g$ is a Zoll Riemannian metric with unit-speed geodesics of minimal period $\ell>0$, then $c_g(\alpha_m)=c_g(\beta_m)=m^2\ell^2$ for all $m\in\N$. Conversely, we will prove the following theorem.

\begin{thm}
\label{t:every_point}
Let $M$ be a manifold homeomorphic to an odd dimensional sphere $S^n$, $n\geq 3$, and $g$ a Riemannian metric on $M$.
If $c_g(\alpha_m)=c_g(\beta_m)$ for some $m\geq 1$, then for each $q\in M$ there exists a (possibly iterated) closed geodesic $\gamma\in\crit(E)$ with $\gamma(0)=q$ and $E(\gamma)=c_g(\alpha_m)$.
\end{thm}

Our second result provides a min-max characterization of Zoll Riemannian metrics on simply connected spin closed manifolds. The statement requires a new finite dimensional reduction of the variational settings for the energy, which goes as follows. Let $(M,g)$ be a closed Riemannian manifold of dimension $n\geq2$, with associated energy functional $E:\Lambda M\to[0,\infty)$. We denote by $\rho=\injrad(M,g)>0$ the injectivity radius, and by $d:M\times M\to[0,\infty)$ the Riemannian distance. For each $\delta\in(0,\rho)$ and  $k\in\N$, we consider the space 
\begin{align*}
\Upsilon M = \Upsilon_{\delta,k}M:=\left\{ \qq=(q_0,...,q_{k-1})\in M\times...\times M\ 
\left|\  
  \begin{array}{@{}l@{}}
    d(q_0,q_1)=\delta, \vspace{5pt} \\ 
    \displaystyle\sum_{i\in\Z_k\setminus\{0\}}\!\!\! d(q_i,q_{i+1})^2<\rho^2
  \end{array}
\right.\right\}.
\end{align*}
For each $\qq\in\Upsilon M$, we define the quantities
\begin{align}
\label{e:sigma}
\sigma(\qq)^2&:=(k-1)\!\!\!\sum_{i\in\Z_k\setminus\{0\}}\!\!\!d(q_i,q_{i+1})^2>0,\\
\label{e:tau_0}
\tau_0(\qq)&:=0,\\
\label{e:tau_1}
\tau_1(\qq)&:=\frac{\delta}{\delta+\sigma(\qq)},\\
\label{e:tau_i}
\tau_{i}(\qq)&:=\tau_1(\qq)+ \frac{(i-1)(1-\tau_1(\qq))}{k-1},\quad\forall i=2,...,k,
\end{align}
so that $0=\tau_0<\tau_1<...<\tau_k=1$.
We consider $\Upsilon M$ as a finite dimensional submanifold of $\Lambda M$, by identifying each $\qq\in\Upsilon M$ with the unique periodic curve $\gamma_{\qq}\in\Lambda M$ such that, for each $i=0,...,k-1$, the restriction $\gamma_{\qq}|_{[\tau_i(\qq),\tau_{i+1}(\qq)]}$ is the unique shortest geodesic joining $q_i$ and $q_{i+1}$.

A feature of $\Upsilon M$ that was missing in $\Lambda M$ is the smooth evaluation map 
\begin{align}
\label{e:evaluation}
\Ev: \Upsilon M\to SM,
\qquad
\Ev(\qq)
:=
\exp_{q_0}^{-1}(q_1),
\end{align}
where we have denoted by $SM$ the unit tangent bundle of $(M,\delta^{-2}g)$. This map is injective in cohomology (Lemma~\ref{l:injectivity}). We choose two generators
\begin{align}
\label{e:omega}
&\omega\in \Ev^*(H^{2n-1}(SM))\cong\Z,\\
\label{e:alpha}
&\alpha\in H^{i(M)}(\Lambda M,\Lambda M^{<4\rho^2})\cong H^{i(M)}(\Lambda M,M)\cong\Z,
\end{align}
We will always fix a sufficiently small parameter $\delta\in(0,\rho)$ and a sufficiently large $k\in\N$ for the space $\Upsilon M=\Upsilon_{\delta,k}M$ so that, according to Lemma~\ref{l:cohomology}, $\omega\smile j^*\alpha\neq 0$ in $H^{i(M)}(\Upsilon M,\Upsilon M^{<4\rho^2})$, where $j:(\Upsilon M,\Upsilon M^{<4\rho^2})\hookrightarrow(\Lambda M,\Lambda M^{<4\rho^2})$ is the inclusion. For each $b>0$, we set $\Upsilon M^{<b}:=\Lambda M^{<b}\cap\Upsilon M$, and denote by 
$j_{b}:(\Upsilon M^{<b},\Upsilon M^{<4\rho^2})
\hookrightarrow
(\Upsilon M,\Upsilon M^{<4\rho^2})$
the inclusion. Given a non-trivial cohomology class $\mu\in H^d(\Upsilon M,\Upsilon M^{<4\rho^2})$ the associated min-max 
\begin{align}
\label{e:minmax_delta_k}
c_g(\mu)=c_g(-\mu):=\inf\{b>0\ |\ j_b^*\mu\neq 0\}
\end{align}
is a critical value of the restricted energy $E|_{\Upsilon M}$, see Section~\ref{e:loop_space}. We can now state our second main theorem.

\begin{thm}
\label{t:main}
Let $M$ be a simply connected spin closed manifold of dimension $n\geq 2$ admitting a simple Zoll Riemannian metric, and $g$ a Riemannian metric on $M$. Then $c_g(j^*\alpha)=c_g(\omega\smile j^*\alpha)=:\ell^2$ if and only if $g$ is Zoll and the unit-speed geodesics of $(M,g)$ have minimal period $\ell$.
\end{thm}

In this theorem, if we drop the assumptions of $M$ being simply connected and spin, it is still true that $c_g(j^*\alpha)=c_g(\omega\smile j^*\alpha)=:\ell^2$ if $g$ is Zoll with unit-speed geodesics of minimal period $\ell$; however, if $c_g(j^*\alpha)=c_g(\omega\smile j^*\alpha)=:\ell^2$, our proof would only imply that $g$ is Besse and either $\ell$ or $\ell-2\delta$ is a common period of the unit-speed closed geodesics.

Actually, the only simply connected closed manifolds that are not spin and admit a Besse Riemannian metric have the same integral cohomology of even dimensional complex projective spaces $\C\PP^{2m}$. Therefore, Theorem~\ref{t:main} applies to $S^n$, $\C \PP^{2m+1}$, $\HH \PP^n$, and $\mathrm{Ca}\PP^2$. The theorem would also apply to all those closed manifolds admitting a simple Zoll Riemannian metric and having the integral cohomology of $S^n$, $\C \PP^{2m+1}$, $\HH \PP^n$, or $\mathrm{Ca}\PP^2$; however, as we already mentioned, there is no known example of simply connected, spin, closed manifold different from $S^n$, $\C \PP^{2m+1}$, $\HH \PP^n$, and $\mathrm{Ca}\PP^2$, and admitting a Besse Riemannian metric.

\subsection{Organization of the paper}
In Section~\ref{s:preliminaries} we provide some background on the energy functional of Besse Riemannian manifolds. In Section~\ref{s:covering} we prove Theorem~\ref{t:every_point}. In Section~\ref{e:loop_space} we study the variational theory of the energy functional in the finite dimensional loop space $\Upsilon M$. Finally, in Section~\ref{s:Zoll} we prove Theorem~\ref{t:main}.

\footnotesize
\subsection*{Acknowledgments}
Marco Mazzucchelli is grateful to Viktor Ginzburg, Basak G\"urel, and Nancy Hingston for discussions concerning the topic of this paper.
This material is based upon work supported by the National Science Foundation under Grant No.~DMS-1440140 while Marco Mazzucchelli was in residence at the Mathematical Sciences Research Institute in Berkeley, California, during the Fall 2018 semester. Stefan Suhr is supported by the SFB/TRR 191 ``Symplectic Structures in Geometry, Algebra and Dynamics'', funded by the Deutsche Forschungsgemeinschaft.

\normalsize

\section{Preliminaries}
\label{s:preliminaries}

Let $E:\Lambda M\to[0,\infty)$ be the energy functional of a closed Besse manifold $(M,g)$ of dimension $n\geq 2$. By a theorem due to Wadsley \cite{Wadsley:1975sp}, all unit-speed geodesics have a minimal common period $\ell>0$. In particular, each critical manifold 
\[K^m:=\crit(E)\cap E^{-1}(m^2\ell^2),\] 
where $m\in\N=\{1,2,3,...\}$, is diffeomorphic to the unit tangent bundle $SM$. 
As usual, we denote by $\ind(E,K^m)$ the Morse index of $E$ at any $\gamma\in K^m$, which is the number of negative eigenvalues of the symmetric operator associated to the Hessian $\diff^2E(\gamma)$. We also denote by $\nul(E,K^m)=\dim\ker\diff^2E(\gamma)-1$ the Morse nullity of $E$ at any $\gamma\in K^m$. Both $\ind(E,K^m)$ and $\nul(E,K^m)$ are independent of the choice of $\gamma$ within $K^m$.
The nullity of a closed geodesic is always bounded from above by $2n-2$; since $\dim(K^m)=2n-1\leq\nul(E,K^m)+1$, we readily infer that $\nul(E,K^m)=2n-2=\dim(K^m)-1$. Therefore, each critical manifold $K^m$ is non-degenerate, meaning that the restriction of the energy functional to the fibers of its normal bundle has non-degenerate Hessian.  This, together with Bott's iteration theory \cite{Bott:1956sp}, implies that
\begin{equation}
\label{e:iter_idx}
\begin{split}
\ind(E ,K^m) & = m\,\ind(E ,K) + (m-1)(n-1),\\
\nul(E, K^m) & =2n-2, 
\end{split}
\end{equation}
see \cite[Eq.~(13.1.1)]{Goresky:2009fq}. Actually, Wilking \cite{Wilking:2001pi} showed that the energy functional $E$ is always Morse-Bott when $(M,g)$ is Besse, meaning that the critical points are organized in critical manifolds  $J\subset \crit(E)$ such that $\nul(E,J)=\dim(J)-1$. Moreover, Radeschi and Wilking \cite[page 941]{Radeschi:2017dz} proved that the minimal index
\begin{align*}
i(M):=\min\big\{\ind(E,\gamma)\ \big|\ \gamma\in\crit(E)\cap E^{-1}(0,\infty)\big\}
\end{align*}
is independent of the choice of a Besse Riemannian metric on $M$, and indeed only depends on the integral cohomology ring of $M$.

Now, let us assume that $(M,g)$ is an orientable Zoll Riemannian manifold. This readily implies that $i(M)=\ind(E,K)$ and
\begin{align*}
 \crit(E)\cap E^{-1}(0,\infty) = \bigcup_{m\geq 1} K^m.
\end{align*}
A result of Goresky and Hingston \cite[Theorem~13.4(1)]{Goresky:2009fq} implies that, if one further assumes that $g$ is simple Zoll, the energy functional $E$ is perfect for the integral singular homology, and equivalently for the integral singular cohomology. This means that, for every integer $m\geq1$ and for all $\epsilon>0$ small enough,  the  homomorphism
\begin{align*}
H^*(\Lambda M,\Lambda M^{<m^2\ell^2})
\ttoup^{\incl^*} 
H^*(\Lambda M^{<m^2\ell^2+\epsilon},\Lambda M^{<m^2\ell^2})
\end{align*}
is surjective. Moreover, by \cite[Proposition~13.2]{Goresky:2009fq}, the negative bundle of every critical manifold $K^m$ is oriented, which implies that
\begin{align*}
H^*(\Lambda M^{<m^2\ell^2+\epsilon},\Lambda M^{<m^2\ell^2})
\cong
H^{*-\ind(E ,K^m)}(SM).
\end{align*}
This, together with a usual gradient flow argument from Morse Theory, implies that cohomology of the free loop space $\Lambda M$ relative to the constant loops $M\subset \Lambda M$ is given by~\eqref{e:cohomology_loop_space}. Actually, for any Zoll Riemannian metric $g$ on $M$, the Morse index formula in~\eqref{e:cohomology_loop_space} becomes $\ind(E,K^m)=m\,i(M)+(m-1)(n-1)$. This readily implies that, on any closed manifold $M$ admitting a simple Zoll Riemannian metric, the energy functional $E:\Lambda M\to[0,\infty)$ of any Zoll Riemannian metric $g$ on $M$ is perfect even if $g$ is not simple Zoll.

\section{A min-max condition for covering with closed geodesics}
\label{s:covering}

Let $M$ be a manifold homeomorphic to an odd dimensional sphere $S^n$, $n\geq 3$. The inclusion $M\subset\Lambda M$ of the constant loops admits the evaluation map \[\ev:\Lambda M\to M,\qquad \ev(\gamma)=\gamma(0)\] as left inverse (this map should not be confused with the velocity evaluation map defined in~\eqref{e:evaluation}). This implies that the cohomology homomorphism 
\begin{align*}
 \ev^*:H^*(M)\hookrightarrow H^*(\Lambda M)
\end{align*}
is injective. 
Ziller's computation \cite{Ziller:1977rp} of the homology of the free loop space of the compact rank-one symmetric spaces gives 
\begin{align*}
H^d(\Lambda M,M)\cong
\left\{
\begin{array}{ll}
    \Z, & \mbox{if }d=m(n-1)\mbox{ or }d=m(n-1)+n,\mbox{ for }m\in\N\cup\{0\}, \\ 
    0, &  \mbox{otherwise}.
  \end{array}
\right.
\end{align*}
In particular, we have two non-trivial generators $\alpha_m\in H^{(2m-1)(n-1)}(\Lambda M,M)$ and $\beta_m\in H^{2m(n-1)+1}(\Lambda M,M)$, as claimed in~\eqref{e:classes_thm1}.

\begin{lem}\label{l:subordinated}
For each $m\in\N$, we have \[\alpha_m\smile\ev^*\nu=\beta_m\] for a suitable generator $\nu\in H^n(M)$.
\end{lem}

\begin{proof}
Let $g$ be a Zoll Riemannian metric on $M$, and $E:\Lambda M\to[0,\infty)$ the associated energy functional. If $K=\crit(E)\cap E^{-1}(\ell^2)$ is the critical manifold of the prime closed geodesics, the other critical manifolds are $M=E^{-1}(0)$ and
\begin{align*}
 K^m=\{\gamma^m\ |\ \gamma\in K\},\qquad
 \forall m\in\N.
\end{align*}
Here, as usual, we have denoted by $\gamma^m\in\Lambda M$ the $m$-th iterate of $\gamma$, which is defined by $\gamma^m(t)=\gamma(mt)$. The associated critical values are $m^2\ell^2=E(K^m)$. Let $G$ be any complete Riemannian metric on $\Lambda M$ (for instance the usual $W^{1,2}$-one induced by $g$). We denote by $\pi:N_m\to K^m$ the negative bundle of $K^m$, which is an orientable vector bundle of rank $\ind(E,K^m)$ whose fibers $\pi^{-1}(\gamma)$ are the negative eigenspaces of the  self-adjoint Fredholm operator $H_\gamma$ on $\Tan_\gamma\Lambda M$ defined by
$G(H_\gamma \cdot,\cdot)=\diff^2E(\gamma)$.
By means of the exponential map of $(\Lambda M,G)$, we can see the total space $N_m$ as a submanifold of $\Lambda M$ containing the critical manifold $K^m$ in its interior. Since $E$ is a Morse-Bott function, for all $\epsilon>0$ small enough the inclusion induces a cohomology isomorphism
\begin{align}
\label{e:local_homology}
H^*(\Lambda M^{<m^2\ell^2+\epsilon} , \Lambda M^{<m^2\ell^2})
\ttoup^{\incl_*}_{\cong}
H^*(N_m,\partial N_m).
\end{align}
The critical manifold $K^m$ is homeomorphic to $SM$ via the map $\gamma\mapsto\dot\gamma(0)/\|\dot\gamma(0)\|_g$. Since both the Euler characteristic $\chi(M)=\chi(S^n)$ and the cohomology group $H^1(M)\cong H^1(S^n)$ vanish, the Gysin sequence of $SM$ implies that $\pi^*:H^n(M)\rightarrow H^n(SM)$ is an isomorphism. Therefore, the evaluation map induces a cohomology isomorphism
\begin{align*}
\ev|_{K^m}^*:H^n(M)\toup^{\cong} H^n(K^m).
\end{align*}
Since the inclusion $K^m\subset N_m$ is a homotopy equivalence, we also have a cohomology isomorphism
\begin{align*}
\ev|_{N_m}^*:H^n(M)\toup^{\cong} H^n(N_m),
\end{align*}
which fits into the commutative diagram 
\begin{equation*}
\begin{split}
  \xymatrix{
      H^n(\Lambda M) \ar[rr]^{\incl^*} & & H^n(N_m)   \\ \\
 \Big. H^n(M)\ar[uurr]_{\ev|_{N_m}^*}^{\cong} \ar@{^{(}->}[uu]^{\ev^*}   & &
 }   
\end{split}
\end{equation*}
We set $\I:=\{(2m-1)(n-1),2m(n-1)+1\}$.
The index formulas~\eqref{e:iter_idx}, together with $i(S^n)=n-1$, imply that
\begin{align*}
H^d(N_m,\partial N_m)\cong H^{d-\ind(E,K^m)}(K^m)\cong H^{d-(2m-1)(n-1)}(SS^n)\cong\Z,
\qquad \forall d\in\I.
\end{align*}
This, together with the fact that $E$ is a perfect functional and that the relative cohomology groups $H^*(\Lambda M,M)$ have rank at most rank 1 in each degree, implies that the inclusion  induces the cohomology isomorphism
\begin{align}
\label{e:incl_*}
\kappa_1^*: H^d(\Lambda M,\Lambda M^{<m^2\ell^2})
\ttoup^{\incl^*}_{\cong}
H^d(N_m,\partial N_m),
\qquad
\forall d\in\I.
\end{align}
Let $\nu$ be a generator of $H^n(M)$. 
We set
\begin{align*}
\mu:=\ev^*\nu\in H^n(\Lambda M),
\qquad
\mu':=\ev|_{N_m}^*\nu\in H^n(N_m).
\end{align*}
Let  $\alpha'\in H^{(2m-1)(n-1)}(N_m,\partial N_m)$ be the Thom class of the orientable vector bundle $N_m\to K^m$ corresponding to an arbitrary orientation. The Thom isomorphism implies that $\alpha'\smile\mu'$ is a generator of $H^{2m(n-1)+1}(N_m,\partial N_m)$.  Once again, since $E$ is a perfect functional, we have an isomorphism
\begin{align*}
\kappa_2^*:H^d(\Lambda M,\Lambda M^{<m^2\ell^2})
\ttoup^{\incl^*}_{\cong}
H^d(\Lambda M,M),
\qquad
\forall d\in\I,
\end{align*}
and we infer
\begin{align*}
(\kappa_2^*)^{-1}\beta_m
&=
(-1)^h
(\kappa_1^*)^{-1}(\alpha'\smile\ev|_{N_m}^*\nu)\\
&=
(-1)^h
(\kappa_1^*)^{-1}\alpha' \smile \ev^*\nu\\
&=
(-1)^h
(\kappa_2^*)^{-1}\alpha_m \smile \ev^*\nu\\
&=
(-1)^h
(\kappa_2^*)^{-1}(\alpha_m \smile \ev^*\nu).
\end{align*}
for some $h\in\{0,1\}$. Up to replacing $\nu$ with $-\nu$, we can assume that $h=0$.
\end{proof}

\begin{proof}[Proof of Theorem~\ref{t:every_point}]
Let us assume by contradiction that $\ell^2:=c_g(\alpha_m)=c_g(\beta_m)$ for some $m\in\N$, but that for some $q\in M$ there is no $\gamma\in\crit(E)$ of energy $E(\gamma)=\ell^2$ with $\gamma(0)=q$. Under this latter assumption, the open subset
$
U:=\Lambda M\setminus \ev^{-1}(q)
$
is a neighborhood of the critical set $\crit(E)\cap E^{-1}(\ell^2)$. By Lemma~\ref{l:subordinated}, $\beta_m=\alpha_m\smile\ev^*\nu$ for some generator $\nu$ of $H^n(M)$, and therefore the classical Lusternik-Schnirelmann theorem (see, e.g., \cite[Theorem~1.1]{Viterbo:1997pi} for a modern account) implies that $(\ev^*\nu)|_{U}\neq0$ in $H^n(U)$. 
Now, consider the commutative diagram
\begin{equation*}
  \xymatrix{
  H^n(M) \ar[rr]^{\ev^*}\ar[dd]_{\incl^*}  && H^n(\Lambda M) \ar[dd]_{\incl^*}\\\\
 H^n(M\setminus\{q\}) \ar[rr]^{\ev|_U^*}  & & H^n(U)
 }   
\end{equation*}
Since $M$ has dimension $n$, the punctured manifold $M\setminus\{q\}$ has trivial cohomology group $H^n(M\setminus\{q\})$. This, together with the above commutative diagram, implies that $(\ev^*\nu)|_{U}=\ev^*(\nu|_{M\setminus\{q\}})=0$, contradicting  Lusternik-Schnirelmann theorem.
\end{proof}

\section{The finite dimensional loop space $\Upsilon M$}
\label{e:loop_space}

Let $(M,g)$ be a closed Riemannian manifold of dimension $n\geq2$ with associated Riemannian distance $d:M\times M\to[0,\infty)$, injectivity radius $\rho=\injrad(M,g)>0$, and energy functional $E:\Lambda M\to[0,\infty)$. 
For each $\delta\in(0,\rho)$ and $k\in\N$, we consider the space $\Upsilon M=\Upsilon_{\delta,k}M$, which we identify with a subspace of $\Lambda M$ as we explained in the introduction. 
The restriction of the energy functional $E_{\delta,k}:=E|_{\Upsilon M}$ can be expressed as
\begin{align*}
E_{\delta,k}(\qq)
&=
\int_0^1 \|\dot\gamma_{\qq}(t)\|^2_g\,\diff t
=
\sum_{i=0}^{k-1} \frac{d(q_i,q_{i+1})^2}{\tau_{i+1}(\qq)-\tau_i(\qq)}\\
&=
\frac{\delta^2}{\tau_1(\qq)} 
+ 
\frac{k-1}{1-\tau_1(\qq)}
\sum_{i\in\Z_k\setminus\{0\}} d(q_i,q_{i+1})^2\\
&=
\left(\delta+\sqrt{(k-1)\sum_{i\in\Z_k\setminus\{0\}} d(q_i,q_{i+1})^2}\right)^2
.
\end{align*}
Here, the times $0=\tau_0(\qq)<...<\tau_{k}(\qq)=1$ are those defined in Equations \eqref{e:tau_0}, \eqref{e:tau_1} and \eqref{e:tau_i}.
For each $i\in\Z_k$, we define $v_i^{\pm}(\qq)\in\Tan_{q_i}M$ by
\begin{align*}
v_i^{\pm}(\qq):=\dot\gamma_{\qq}(\tau_i(\qq)^{\pm}). 
\end{align*}
The choice of $\tau_1$ that we made in~\eqref{e:tau_1} is such that, for all $\qq\in\crit(E_{\delta,k})$, the corresponding curve $\gamma_{\qq}$ has constant speed (even though $\dot\gamma_{\qq}$ may not be smooth at times $\tau_0(\qq)=0$ and $\tau_1(\qq)$). More precisely, we have the following statement.

\begin{prop}
\label{p:critical_points}
The critical points of $E_{\delta,k}$ are precisely those $\qq\in\Upsilon M$ such that $v_{0}^{-}(\qq)\in\{v_{0}^+(\qq),-v_{0}^+(\qq)\}$, $v_{1}^{+}(\qq)\in\{v_{1}^-(\qq),-v_{1}^-(\qq)\}$, and $v_i^-(\qq)=v_i^+(\qq)$ for all $i\in\Z_k\setminus\{0,1\}$. 
\end{prop}

\begin{proof}
Consider the functional
\begin{gather*}
F:\underbrace{M\times...\times M}_{\times k}\times(0,1)\to[0,\infty),
\\
F(\qq,\tau)=
\frac{1}{\tau} d(q_0,q_1)^2
+ 
\frac{k-1}{1-\tau}
\sum_{i\in\Z_k\setminus\{0\}} d(q_i,q_{i+1})^2,
\end{gather*}
which is smooth on the subset $U\subset M^{\times k}\times (0,1)$ of all those points $(\qq,\tau)$ such that $d(q_i,q_{i+1})<\rho$ for all $i\in\Z_k$. Notice that, for all $\qq\in\Upsilon M$, we have $E(\gamma_{\qq})=F(\qq,\tau_1(\qq))$, and one can easily verify that $\tau_1(\qq)$ is the unique critical point and the global minimizer of the function $\tau\mapsto F(\qq,\tau)$. In the following, for each $\tau\in(0,1)$ we set
\begin{align*}
 F_\tau:=F(\cdot,\tau).
\end{align*}

We denote by $SM$ the unit tangent bundle of $(M,\delta^{-2}g)$, that is,
\begin{align*}
SM=\big\{ (q,v)\in\Tan M\ \big|\ \|v\|_g=\delta \big\}.
\end{align*}
The space $\Upsilon M$ is diffeomorphic to the space $\Upsilon' M$ of those 
\begin{align*}
\qq'=(q_0,v_0,q_2,...,q_{k-2},q_{k-1})\in SM\times\underbrace{M\times...\times M}_{\times k-2}, 
\end{align*}
such that, if we set $q_1:=\exp_{q_{0}}(v_0)$,  we have 
\begin{align*}
\sum_{i\in\Z_k\setminus\{0\}} \!\!\! d(q_i,q_{i+1})^2<\rho^2.
\end{align*}
The explicit diffeomorphism is 
\begin{align*}
\iota:\Upsilon' M\toup^{\cong}\Upsilon M,
\qquad
\iota(q_0,v_0,q_2,...,q_{k-1})=\qq=(q_0,q_1,q_2,...,q_{k-1}),
\end{align*}
and we have 
\[
\frac{v_0}{\|v_0\|_g}=\frac{\dot\gamma_{\qq}(0^+)}{\|\dot\gamma_{\qq}(0^+)\|_g}.
\] 
We consider the submersion $Q:SM\to M$, $Q(q,v)=\exp_q(v)$. The differential of $\iota$ is given by
\begin{align*}
\diff\iota(\qq')(z,\ww)=(\diff\pi(q_{0},v_{0})z,\diff Q(q_{0},v_{0})z,\ww),
\end{align*}
where $\ww=(w_2,...,w_{k-2})\in\Tan_{q_2}M\times...\times\Tan_{q_{k-1}}M$, and $z\in\Tan_{(q_0,v_0)}SM$. We set
\begin{align*}
\qq=(q_0,q_1,q_2,...,q_{k-1}):=\iota(\qq'),
\qquad
\tau_i:=\tau_i(\qq),
\qquad
v_i^{\pm}:=\dot\gamma_{\qq}(\tau_i^\pm),
\end{align*}
so that
\begin{equation}
\label{e:diff_E_iota}
\begin{split}
 \diff(F_{\tau_1}\circ\iota)(\qq')(\ww,z)
 =
 &
 \sum_{i\in\Z_k\setminus\{0,1\}} 2 g( v_i^- - v_i^+, w_i)\\
 & + 2 g( v_1^- - v_1^+, \diff Q(q_{0},v_{0})z)\\
 & + 2 g( v_{0}^- - v_{0}^+, \diff \pi(q_{0},v_{0})z). 
\end{split}
\end{equation}
By Equation~\eqref{e:diff_E_iota}, $\diff(F_{\tau_1}\circ\iota)(\qq')(0,\ww)=0$ for all $\ww$ if and only if $v_i^- = v_i^+$ for all $i\in\Z_k\setminus\{0,1\}$.

Notice that 
\[\diff Q(q_{0},v_{0})\big(\ker\diff\pi(q_{0},v_{0})\big)=\mathrm{span}\{v_{1}^-\}^{\bot}.\]
Therefore, by Equation~\eqref{e:diff_E_iota}, $\diff(F_{\tau_1}\circ\iota)(\qq')(z,0)=0$ for all $z\in\ker\diff\pi(q_{0},v_{0})$ if and only if  $v_1^- - v_1^+ \bot \mathrm{span}\{v_{1}^-\}^{\bot}$, that is, $v_{1}^{+}\in\mathrm{span}\{v_{1}^-\}$.

Now, fix an arbitrary tangent vector 
\[v\in\mathrm{span}\{v_{0}\}^{\bot}=\mathrm{span}\{v_{0}^+\}^{\bot},\] 
and choose any smooth curve $\zeta:(-\epsilon,\epsilon)\to M$ such that $\zeta(0)=q_{0}$, $\dot\zeta(0)=v$, and $d(\zeta(t),q_1)=\delta$ for all $t\in(-\epsilon,\epsilon)$. We set $\xi(t):=\exp_{\zeta(t)}^{-1}(q_1)$, and \[z:=\tfrac{\diff}{\diff t}|_{t=0}(\zeta(t),\xi(t)).\] Notice that
\begin{align*}
\diff\pi(q_{0},v_{0})z=v,\qquad \diff Q(q_{0},v_{0})z=0.
\end{align*}
Therefore, by Equation~\eqref{e:diff_E_iota}, $\diff(F_{\tau_1}\circ\iota)(\qq')(z,0)=0$ for all $z$ of this form if and only if $v_{0}^- - v_{0}^+ \bot \mathrm{span}\{v_{0}^+\}^{\bot}$, that is, $v_{0}^{-}\in\mathrm{span}\{v_{0}^+\}$.

It remains one last case in order to cover all the possible choices of tangent vectors $z\in\Tan_{(q_{0},v_{0})}SM$, namely when $z$ is the value of the geodesic vector field at $(q_{0},v_{0})$. In this case, 
\begin{align*}
\diff\pi(q_{0},v_{0})z=v_{0}=\frac{\delta}{\|v_0^-\|_g}v_0^+,\qquad \diff Q(q_{0},v_{0})z=\frac{\delta}{\|v_1^-\|_g} v_1^-=\frac{\delta}{\|v_0^-\|_g} v_1^-.
\end{align*}
Therefore, by Equation~\eqref{e:diff_E_iota}, $\diff(F_{\tau_1}\circ\iota)(\qq')(z,0)=0$ if and only if
\begin{align*}
g( v_1^- - v_1^+, v_1^-) - g( v_{0}^- - v_{0}^+, v_{0}^+)=0.
\end{align*}
Since $\|v_1^-\|_g=\|v_{0}^+\|_g$, this latter equation is verified if and only if
\begin{align}
\label{e_redundant}
g( v_1^+, v_1^-) = g( v_{0}^-, v_{0}^+).
\end{align}
Notice however that condition~\eqref{e_redundant} for the critical points of $E_{\delta,k}$ is redundant: indeed, any point $\qq\in\Upsilon M$ such that $v_{0}^{-}\in\mathrm{span}\{v_{0}^+\}$, $v_{1}^{+}\in\mathrm{span}\{v_{1}^-\}$, $v_i^-=v_i^+$ for all $i\in\Z_k\setminus\{0,1\}$, and $g( v_1^+, v_1^-) \neq g( v_{0}^-, v_{0}^+)$ would define a geodesic cusp $\gamma_{\qq}$ and thus violate the uniqueness of the solution to the geodesic equation.

Summing up, we have proved that $\qq\in\crit(E_{\delta,k})$ if and only if $v_{0}^{-}\in\mathrm{span}\{v_{0}^+\}$, $v_{1}^{+}\in\mathrm{span}\{v_{1}^-\}$, $v_i^-=v_i^+$ for all $i\in\Z_k\setminus\{0,1\}$. In this case, we have
\begin{align*}
\sigma=\sigma(\qq)=\int_{\tau_1}^1 \|\dot\gamma_{\qq}\|_g\,\diff t,
\end{align*}
and therefore
\begin{align*}
\|v_0^+\|_g=\frac{\delta}{\tau_1}=\delta+\sigma=\int_{0}^1 \|\dot\gamma_{\qq}\|_g\,\diff t = \tau_1\|v_0^+\|_g+ (1-\tau_1)\|v_1^+\|_g,
\end{align*}
which implies that $\|v_0^+\|_g=\|v_1^+\|_g$.
\end{proof}

\begin{figure}
\begin{center}
\begin{small}
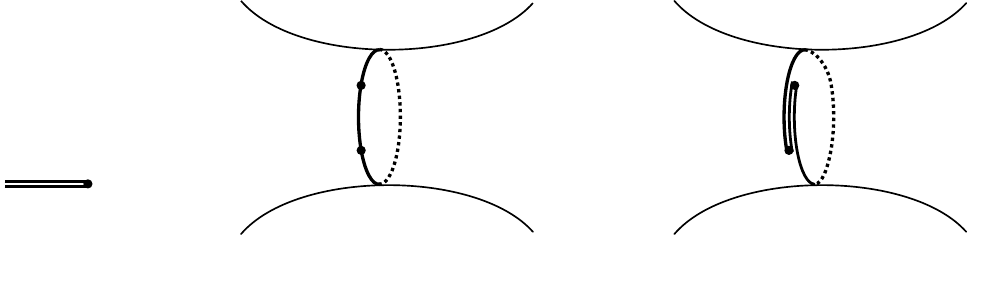 
\caption{\textbf{(a)} A global minimizer of $E_{\delta,k}$. \textbf{(b)} A critical point of $E_{\delta,k}$ corresponding to a closed geodesic. \textbf{(c)} The ``zig-zag'' critical point of $E_{\delta,k}$ corresponding to the same closed geodesic.}
\label{f:critical_points}
\end{small}
\end{center}
\end{figure}

Proposition~\ref{p:critical_points} shows that, beside the global minimizers $E_{\delta,k}^{-1}(4\delta^2)$ (Figure~\ref{f:critical_points}(a)) there are two other kinds of critical points of $E_{\delta,k}$: the closed geodesics smoothly parametrized with constant speed (Figure~\ref{f:critical_points}(b)), and the closed geodesics pa\-ram\-e\-trized with constant speed but non-smoothly with a zig-zag at times $\tau_0=0$ and $\tau_1$ (Figure~\ref{f:critical_points}(c)). If $\qq'\in\crit(E_{\delta,k})$ corresponds to a smoothly parametrized closed geodesic $\gamma_{\qq'}\in\crit(E)$ and $\qq''\in\crit(E_{\delta,k})$ corresponds to the same closed geodesic parametrized with a zig-zag, their energies are related by
\begin{align}
\label{e:energy_zigzag}
E_{\delta,k}(\qq'')^{1/2}=E_{\delta,k}(\qq')^{1/2}+2\delta.
\end{align}
We partition the critical point set $\crit(E_{\delta,k})$ as the disjoint union
\begin{align*}
\crit(E_{\delta,k})
=
E_{\delta,k}^{-1}(4\delta^2)
\cup
K'
\cup
K'',
\end{align*}
where $K'=\crit(E)\cap E^{-1}(0,\infty)$, while $K''$ contains zig-zag closed geodesics.

The functional setting of the energy $E_{\delta,k}:\Upsilon M\to[4\delta^2,\infty)$ is suitable for Morse theory. Indeed, $E_{\delta,k}$ can be continuously extended to the boundary $\partial \Upsilon M\subset M\times...\times M$, and we have
\begin{align}
\label{e:max_E_delta_k}
E_{\delta,k}|_{\partial \Upsilon M}
\equiv 
\sup E_{\delta,k}
=
\big( \delta + \sqrt{(k-1)}\rho \big)^2.
\end{align}
In particular, every sublevel set $\Upsilon M^{\leq b}$, for $b<\big( \delta + \sqrt{(k-1)}\rho \big)^2$, is a compact subset of $\Upsilon M$. Therefore, the classical min-max theorem is available in this setting: for each non-trivial cohomology class $\mu\in H_*(\Upsilon M,\Upsilon M^{<4\rho^2})$, the min-max value $c_g(\mu)$ defined in~\eqref{e:minmax_delta_k} is a critical value of $E_{\delta,k}$. 
Actually, each closed geodesic $\gamma\in\crit(E)\cap E^{-1}(\ell^2)$, with $\ell>0$, is contained in $\Upsilon M=\Upsilon_{\delta,k}M$ if and only if
\begin{align*}
k
>
\overline k(\ell,\delta)
:=
1 + \frac{(\ell-\delta)^2}{\rho^{2}} 
\end{align*}
Indeed, if we define 
\begin{align*}
\tau_0 & :=0,\\
\tau_{i} & :=\frac{\delta}{\ell}+(i-1)\frac{ \ell-\delta}{ (k-1)\ell },
\qquad i=1,...,k-1,
\end{align*}
we readily verify that $\qq=(\gamma(\tau_0),...,\gamma(\tau_{k-1}))$ belongs to $\Upsilon M$, and $\gamma_{\qq}=\gamma$.

The following two lemmas compare the Morse indices in the settings $\Lambda M$ and $\Upsilon M$. The reader may skip their rather technical proofs on a first reading.

\begin{lem}
\label{l:indices}
Let $\gamma\in\crit(E)\cap E^{-1}(0,\infty)$ be a closed geodesic. For each $\delta\in(0,\rho)$ and $k> \overline k(E(\gamma)^{1/2},\delta)$, if $\qq\in \crit(E_{\delta,k})$ is such that $\gamma_{\qq}=\gamma$, then
\[\ind(E,\gamma)=\ind(E_{\delta,k},\qq),
\qquad
\nul(E,\gamma)=\nul(E_{\delta,k},\qq).\]
\end{lem}

\begin{proof}
We set $\theta_0:=0$, $\theta_k:=1$, and, for each $i=1,...,k-1$, we choose a time value $\theta_i\in(\tau_i(\qq),\tau_{i+1}(\qq))$ sufficiently close to $\tau_i(\qq)$ so that $d(\gamma(\theta_i),\gamma(\theta_{i+1}))<\rho$ for all $i\in\Z_k$. Notice that
\begin{align*}
0
=
\tau_0(\qq)=\theta_0
<
\tau_1(\qq)<\theta_1
<
\tau_2(\qq)<\theta_2
<
...
<
\tau_k(\qq)=\theta_k
=1.
\end{align*}
We set $q_i':=\gamma(\theta_i)$ and $\qq':=(q_0',...,q_{k-1}')$. We consider the function
\begin{align*}
F:M^{\times k}\to[0,\infty),
\qquad
F(\pp')
=
\sum_{i=0}^{k-1}
\frac{d(p_i',p_{i+1}')^2}{\theta_{i+1}-\theta_i},
\end{align*}
which is smooth on an open neighborhood of $\qq'$. Since $\qq'$ is obtained by sampling the closed geodesic $\gamma$ at times $\theta_i$, it is a critical point of $F$. Since $d(\gamma(\theta_i),\gamma(\theta_{i+1}))<\rho$ and for all $i\in\Z_k$, it is well known that
\begin{align*}
\ind(E,\gamma)=\ind(F,\qq'),
\qquad
\nul(E,\gamma)=\nul(F,\qq'),
\end{align*}
see, e.g., \cite[Theorem~16.2]{Milnor:1963rf}. For each $i\in\Z_k$, we denote by $\Sigma_i\subset\Tan_{q_i'}M$ the hyperplane orthogonal to $\dot\gamma(\theta_i)$. By the definition of the Morse indices, there exist vector subspaces $\V,\W\subset \Sigma_0\times...\times\Sigma_{k-1}$ of dimensions $\ind(E,\gamma)$ and $\ind(E,\gamma)+\nul(E,\gamma)$ respectively such that
\begin{equation}
\label{e:negative_spaces}
\begin{split}
\diff^2F(\qq')[\vv',\vv']<0,&
\qquad \forall \vv'\in\V\setminus\{0\},\\
\diff^2F(\qq')[\vv',\vv']\leq0,&
\qquad \forall \vv'\in\W.
\end{split}
\end{equation}

Now, we choose an open neighborhood $U\subset M^{\times k}$ of $\qq'$ that is small enough so that, for all $\pp'=(p_0',...,p_{k-1}')\in U$, we have 
\begin{align*}
&\delta<d(p_0',p_1')<\rho,\\
&d(p_i',p_{i+1}')<\rho,\qquad \forall i\in\Z_{k}\setminus\{0\}.
\end{align*}
We define $\zeta_{\pp'}\in\Lambda M$ to be the piecewise broken geodesic such that each restriction $\zeta_{\pp'}|_{[\theta_i,\theta_{i+1}]}$ is the shortest geodesic joining $p_i'$ and $p_{i+1}'$. Notice that $\zeta_{\qq'}=\gamma_{\qq}=\gamma$. We set
\begin{align*}
\nu_0(\pp') &:= 0,\\
\nu_1(\pp') &:= \frac{\theta_1 \delta}{d(p_0',p_1')},\\
\nu_i(\pp') &:= \nu_1(\pp') + (i-1)\frac{1-\nu_1(\pp')}{k-1},
\qquad\forall i=2,...,k.
\end{align*}
Notice that $\nu_i(\qq')=\tau_i(\qq)$. Therefore, up to replacing $U$ with a smaller neighborhood of $\qq'$, for each $\pp'\in U$, the curve $\zeta_{\pp'}$ is smooth at each time $\nu_i(\pp')$. This implies that the map
\begin{align*}
\psi:U\to\Upsilon M=\Upsilon_{\delta,k} M,
\qquad
\psi(\pp')=(\zeta_{\pp'}(\nu_0(\pp')),...,\zeta_{\pp'}(\nu_{k-1}(\pp'))).
\end{align*}
is smooth. Notice that $\psi(\qq')=\qq$ and
\begin{align*} 
E_{\delta,k}(\psi(\pp'))\leq F(\pp'),
\qquad
\forall \pp'\in U,
\end{align*}
with equality if $\pp'=\qq'$. This, together with~\eqref{e:negative_spaces}, implies that
\begin{equation}
\label{e:negative_spaces_2}
\begin{split}
\diff^2E_{\delta,k}(\qq)[\diff\psi(\qq')\vv',\diff\psi(\qq')\vv']\leq \diff^2F(\qq')[\vv',\vv']<0,&
\qquad \forall \vv'\in\V\setminus\{0\},\\
\diff^2E_{\delta,k}(\qq)[\diff\psi(\qq')\vv',\diff\psi(\qq')\vv']\leq \diff^2F(\qq')[\vv',\vv']\leq0,&
\qquad \forall \vv'\in\W.
\end{split}
\end{equation}

Each $\vv'=(v_0',...,v_{k-1}')\in\Tan_{\qq'}M^{\times k}$ defines a unique continuous and piecewise smooth vector field $\xi_{\vv'}$ along $\gamma$ such that, for all $i=0,...,k-1$,  $\xi_{\vv'}(\theta_i)=v_i'$ and the restriction $\xi_{\vv'}|_{[\theta_i,\theta_{i+1}]}$ is a Jacobi vector field. Since $d(q_i,q_{i+1})<\rho$, the geodesic $\gamma|_{[\tau_i(\qq),\tau_{i+1}(\qq)]}$ is the shortest one joining $q_i$ and $q_{i+1}$. This readily implies that the map
\begin{align*}
\Psi:\Tan_{\qq'}M^{\times k}\to \Tan_{\qq}M^{\times k},
\qquad
\Psi(\vv')=(\xi_{\vv'}(\tau_0(\qq)),...,\xi_{\vv'}(\tau_{k-1}(\qq)))
\end{align*}
is injective on $\V$ and on $\W$. The differential of $\psi$ at $\qq'$ is given by
\begin{align*}
\diff\psi(\qq')\vv'
=
\Psi(\vv') + \big(\dot\gamma(\tau_0(\qq))\diff\nu_0(\qq')\vv',...,\dot\gamma(\tau_{k-1}(\qq))\diff\nu_{k-1}(\qq')\vv' \big)
\end{align*}
Consider a non-zero $\vv'\in\V\cup\W$, and set $\vv=(v_0,...,v_{k-1}):=\Psi(\vv')$. By the injectivity of $\Psi$, at least one component of $\vv$, say $v_i$, is non-zero. Since both tangent vectors $v_{i-1}'$ and $v_i'$ are orthogonal to $\dot\gamma$, the whole Jacobi field $\xi_{\vv'}|_{[\theta_{i-1},\theta_i]}$ is pointwise orthogonal to $\dot\gamma$, and so is $v_i=\xi_{\vv'}(\tau_i(\qq))$. Therefore, $v_i+\dot\gamma(\tau_{i}(\qq))\diff\nu_{i}(\qq')\vv'$ is non-zero, which shows that the differential $\diff\psi(\qq')$ is injective on both $\V$ and $\W$. This, together with~\eqref{e:negative_spaces_2}, implies that
\begin{align*}
\ind(E_{\delta,k},\qq) & \geq \ind(E,\gamma),\\
\ind(E_{\delta,k},\qq)+\nul(E_{\delta,k},\qq) & \geq \ind(E,\gamma)+\nul(E,\gamma).
\end{align*}
Since $\Upsilon M$ is a subspace of $\Lambda M$, the opposite inequalities hold as well.
\end{proof}

\begin{lem}
\label{l:indices_zigzag}
Let $\gamma\in\crit(E)\cap E^{-1}(0,\infty)$ be a closed geodesic. For each $\delta\in(0,\rho)$ and integer $k> \overline k(E(\gamma)^{1/2}+2\delta,\delta)$, let $\qq'\in \crit(E_{\delta,k})$ be such that $\gamma_{\qq'}=\gamma$, and $\qq''\in \crit(E_{\delta,k})$ be the associated zig-zag critical point, i.e.\ $q_0''=q_0'$, $q_1''=q_1'$, and $E_{\delta,k}(\qq'')^{1/2}=E_{\delta,k}(\qq'')^{1/2}+2\delta$. Then
\begin{align*}
\ind(E_{\delta,k},\qq') & \leq \ind(E_{\delta,k},\qq''),
\\
\ind(E_{\delta,k},\qq')+\nul(E_{\delta,k},\qq') & \leq \ind(E_{\delta,k},\qq'')+\nul(E_{\delta,k},\qq'').
\end{align*}
\end{lem}

\begin{proof}
The proof is somewhat analogous to the one of Lemma~\ref{l:indices}, but requires some extra ingredients. We consider the time values 
$1+\tau_1(\qq') =: \sigma_1
>
\sigma_2
>
...
>\sigma_k :=0$
such that
\begin{align*}
\gamma_{\qq''}(t\tau_{i+1}(\qq'')+(1-t)\tau_i(\qq''))
=
\gamma(t\sigma_{i+1}+(1-t)\sigma_i),
\qquad
\forall t\in[0,1].
\end{align*}
We choose arbitrary values
$
0=:\theta_0<\theta_1<...<\theta_k:=1
$
such that 
\begin{align*}
&\{\theta_1,...,\theta_{k-1}\}
\cap
\{\sigma_1\ \mathrm{mod}\ 1,...,\sigma_{k-1}\ \mathrm{mod}\ 1\}=\varnothing,\\
&d(\gamma_{\qq'}(\theta_0),\gamma_{\qq'}(\theta_1))>\delta,\\
&d(\gamma_{\qq'}(\theta_i),\gamma_{\qq'}(\theta_{i+1}))<\rho,\quad
\forall i=0,...,k-1.
\end{align*}
The function
\begin{align*}
F:M^{\times k}\to[0,\infty),
\qquad
F(\pp)
=
\sum_{i=0}^{k-1}
\frac{d(p_i,p_{i+1})^2}{\theta_{i+1}-\theta_i},
\end{align*}
is smooth on an open neighborhood of $\qq:=(\gamma_{\qq'}(\theta_0),...,\theta_{\qq'}(\theta_{k-1}))$. Since $\qq$ is obtained by sampling the closed geodesic $\gamma$ at times $\theta_i$, it is a critical point of $F$. By Lemma~\ref{l:indices} and \cite[Theorem~16.2]{Milnor:1963rf}, we have
\begin{align*}
\ind(E_{\delta,k},\qq')
=
\ind(E,\gamma)
=
\ind(F,\qq).
\end{align*}
Therefore, if we denote by $\Sigma_i\subset\Tan_{q_i}M$ the hyperplane orthogonal to $\dot\gamma(\theta_i)$, we can find vector subspaces $\V,\W\subset \Sigma_1\times...\times\Sigma_{k-1}$ of dimensions $\ind(E,\gamma)$ and $\ind(E,\gamma)+\nul(E,\gamma)$ respectively such that
\begin{equation}
\label{e:negative_eigespaces_zigzag}
\begin{split}
\diff^2F(\qq)[\vv,\vv]<0,&
\qquad \forall \vv\in\V\setminus\{0\},\\
\diff^2F(\qq)[\vv,\vv]\leq0,&
\qquad \forall \vv\in\W. 
\end{split}
\end{equation}

We choose an open neighborhood $U\subset M^{\times k}$ of $\qq$ that is small enough so that $d(p_0,p_{1})>\delta$ and $d(p_i,p_{i+1})<\rho$ for all $\pp=(p_0,...,p_{k-1})\in U$ and $i\in\Z_{k}\setminus\{0\}$. 
We define $\overline\beta_{\pp}\in\Lambda M$ to be the piecewise broken geodesic such that each restriction $\overline\beta_{\pp}|_{[\theta_i,\theta_{i+1}]}$ is the shortest geodesic joining $p_i$ and $p_{i+1}$. Notice that $F(\pp)=E(\overline\beta_{\pp})$. Moreover, $\overline\beta_{\qq}=\gamma_{\qq'}=\gamma$, which implies
\begin{align*}
F(\qq)=E(\overline\beta_{\qq})=E_{\delta,k}(\qq').
\end{align*}
We set $\overline\nu(\pp):=\delta\theta_1 d(p_0,p_1)^{-1}\in(0,\theta_1)$ and notice that, since the restriction $\overline\beta_{\pp}|_{[0,\theta_1]}$ is a geodesic, we have 
\[
d\big(\overline\beta_{\pp}(0),\overline\beta_{\pp}(\overline\nu(\pp))\big)
=
\delta.
\]
We set
\begin{align*}
\nu(\pp)
:=
\frac
{\delta}
{\delta+\sqrt{(1-\overline\nu(\pp))\int_{\overline\nu(\pp)}^{1} \|\dot{\overline\beta}_{\pp}(t)\|_g^2\,\diff t}}.
\end{align*}
We define $\beta_{\pp}\in\Lambda M$ so that the restrictions $\beta_{\pp}|_{[0,\nu(\pp)]}$ and  $\beta_{\pp}|_{[\nu(\pp),1]}$ are affine reparametrizations of $\overline\beta_{\pp}|_{[0,\overline\nu(\pp)]}$ and  $\overline\beta_{\pp}|_{[\overline\nu(\pp),1]}$ respectively; namely,
\begin{align*}
\beta_{\pp}(t)
:=
\left\{
  \begin{array}{lll}
    \overline\beta_{\pp}\big(t\,\tfrac{\overline\nu(\pp)}{\nu(\pp)}\big), &  & t\in[0,\nu(\pp)], \vspace{5pt}\\ 
    \overline\beta_{\pp}\big(\overline\nu(\pp)+ (t-\nu(\pp))\,\tfrac{1-\overline\nu(\pp)}{1-\nu(\pp)}\big), &  & t\in[\nu(\pp),1].
  \end{array}
\right.
\end{align*}
The choice of this reparametrization guarantees that
\begin{align*}
E(\beta_{\pp})\leq E(\overline\beta_{\pp}) = F(\pp).
\end{align*}
Moreover, $\nu(\qq)=\overline\nu(\qq)=\tau_1(\qq')$ and $\beta_{\qq}=\overline\beta_{\qq}=\gamma_{\qq'}=\gamma$, and therefore
\begin{align*}
E(\beta_{\qq})=F(\qq).
\end{align*}
We define $\overline\alpha_{\pp}:\R/(1+2\nu(\pp))\Z\to M$ by suitably adding a zig-zag to $\beta_{\qq}$ as follows
\begin{align*}
\overline\alpha_{\pp}(t)
:=
\left\{
  \begin{array}{lll}
    \beta_{\pp}(t), &  & t\in[0,\nu(\pp)], \vspace{5pt}\\ 
    \beta_{\pp}(2\nu(\pp)-t), &  & t\in[\nu(\pp),2\nu(\pp)],\vspace{5pt}\\ 
    \beta_{\pp}(t-2\nu(\pp)), &  & t\in[2\nu(\pp),1+2\nu(\pp)],
  \end{array}
\right.
\end{align*}
and we define $\alpha_{\pp}\in\Lambda M$ by $\alpha_{\pp}(t)=\overline\alpha_{\pp}(t\,(1+2\nu(\pp)))$. The energies of $\alpha_{\pp}$ and $\beta_{\pp}$ are related by
\begin{align*}
E(\alpha_{\pp}) = (2\delta + E(\beta_{\pp})^{1/2})^2.
\end{align*}
Moreover, $\alpha_{\qq}=\gamma_{\qq''}$.
We set
\begin{align*}
\eta_i(\pp) &
:= 
\frac{k-i}{k-1}
\left(
1 + \nu(\pp)
\right),
\qquad\forall i=1,...,k,
\end{align*}
and notice that $\eta_i(\qq)=\sigma_i$.
Up to replacing $U$ with a smaller neighborhood of $\qq$, for each $\pp\in U$ the curve $\beta_{\pp}$ is smooth at each time $\eta_i(\pp)$. Therefore, the map
\begin{align*}
\psi:U\to\Upsilon M=\Upsilon_{\delta,k} M,
\qquad
\psi(\pp)=(\beta_{\pp}(0),\beta_{\pp}(\eta_1(\pp)),\beta_{\pp}(\eta_2(\pp)),...,\beta_{\pp}(\eta_{k-1}(\pp)))
\end{align*}
is smooth, and satisfies 
\[\psi(\qq)=\qq''\in\crit(E_{\delta,k}).\] Notice that
$E_{\delta,k}(\psi(\pp))\leq E(\alpha_{\pp})$
with equality if $\pp=\qq$. This, together with the other energy inequalities pointed out so far, provides
\begin{align*}
E_{\delta,k}(\psi(\pp)) & \leq (2\delta+F(\pp)^{1/2})^2,\quad \forall\pp\in U,\\
E_{\delta,k}(\psi(\qq)) & = E_{\delta,k}(\qq'')  = (2\delta+F(\qq)^{1/2})^2.
\end{align*}
This, together with the fact that $\psi(\qq)=\qq''$ and $\qq$ are critical points of $E_{\delta,k}$ and $F$ respectively, implies 
\begin{align*}
\diff^2E_{\delta,k}(\qq'')[\diff\psi(\qq)\vv,\diff\psi(\qq)\vv]
\leq
\frac{2\delta+F(\qq)^{1/2}}{F(\qq)^{1/2}} \,\diff^2 F(\qq)[v,v].
\end{align*}
Therefore, by~\eqref{e:negative_eigespaces_zigzag}, we infer
\begin{align*}
\diff^2E_{\delta,k}(\qq'')[\diff\psi(\qq)\vv,\diff\psi(\qq)\vv]<0,&
\qquad \forall \vv\in\V\setminus\{0\},\\
\diff^2E_{\delta,k}(\qq'')[\diff\psi(\qq)\vv,\diff\psi(\qq)\vv]\leq0,&
\qquad \forall \vv\in\W, 
\end{align*}
which provides the following lower bounds for the Morse indices
\begin{align*}
\ind(E_{\delta,k},\qq'') \geq \dim(\diff\psi(\qq)\V),
\qquad
\ind(E_{\delta,k},\qq'')+\nul(E_{\delta,k},\qq'') \geq \dim(\diff\psi(\qq)\W).
\end{align*}
Finally, the same argument as in the proof of Lemma~\ref{l:indices} implies that $\diff\psi(\qq)$ is injective on both $\V$ and $\W$, i.e.
\begin{align*}
\dim(\diff\psi(\qq)\V)&= \dim(\V) = \ind(E_{\delta,k},\qq'),
\\
\dim(\diff\psi(\qq)\W)&=\dim(\W)=\ind(E_{\delta,k},\qq')+\nul(E_{\delta,k},\qq').
\qedhere
\end{align*}
\end{proof}

\section{Zoll Riemannian metrics}
\label{s:Zoll}

\subsection{The evaluation map on $\Upsilon M$}
\label{s:evaluation}

Let us quickly prove the following property of the evaluation map $\Ev:\Upsilon M\to SM$ defined in~\eqref{e:evaluation}. Here, as before, we denote by $SM$ the unit tangent bundle of $(M,\delta^{-2}g)$.

\begin{lem}
\label{l:injectivity}
For each $b\geq4\delta^2$, the cohomology homomorphism \[\Ev^*:H^*(SM)\hookrightarrow H^*(\Upsilon M^{\leq b})\] is injective.
\end{lem}

\begin{proof}
We define the homeomorphism 
\begin{align*}
\iota:SM\to E_{\delta,k}^{-1}(4\delta^2)\subset\Upsilon M,\qquad \iota(q_0,v_0)=\qq,
\end{align*}
where $\qq=(q_0,...,q_{k-1})$ is the unique element in $E_{\delta,k}^{-1}(4\delta^2)$ such that $\exp_{q_{0}}(v_0)=q_1$.
Namely, $\gamma_{\qq}$ is the periodic curve such that $\tau_1(\qq)=1/2$,  $\dot\gamma_{\qq}(0^+)=2v_{0}$, and 
$\gamma_{\qq}(t)=\gamma_{\qq}( 1-t)$ for all $t\in[0,1/2]$, as in Figure~\ref{f:critical_points}(a). Since the composition $\Ev\circ\iota$ is the identity, the lemma follows for $b=4\delta^2$. The lemma readily follows also for any $b>4\delta^2$, since $E_{\delta,k}^{-1}(4\delta^2)\subset\Upsilon M^{\leq b}$.
\end{proof}

As in~\eqref{e:omega}, we denote by $\omega$ the generator of $\Ev^*(H^{2n-1}(SM))$. The following lemma is the main ingredient for the proof of Theorem~\ref{t:main}.

\begin{lem}
\label{l:main_lemma}
Assume that there exists a cohomology class $\mu\in H^d(\Upsilon M,\Upsilon M^{\leq4\rho^2})$ such that $\omega\smile\mu\neq0$ in $H^{d+2n-1}(\Upsilon M,\Upsilon M^{\leq4\rho^2})$. If 
\[c_g(\mu)=c_g(\omega\smile\mu)=:\ell^2,\] 
then $g$ is a Besse manifold, and either $\ell$ or $\ell-2\delta$ is a common multiple of the periods of the unit-speed closed geodesics of $(M,g)$. Moreover, the critical set \[K:=\crit(E)\cap(E^{-1}(\ell^2)\cup E^{-1}((\ell-2\delta)^2)),\] 
has Morse index $\ind(E,K)\leq d$.
\end{lem}

\begin{proof}
Assume by contradiction that $c_g(\mu)=c_g(\omega\smile\mu)=:\ell^2$, but there exists $(q,v)\in SM$ such that the unit-speed geodesic \[\gamma(t)=\exp_q(tv/\|v\|_g)\] is either not periodic, or it is periodic but neither $\ell$ nor $\ell-2\delta$ are multiples of its minimal period. By~\eqref{e:energy_zigzag}, the condition on $\ell-2\delta$ implies that none of the zig-zag critical points $\qq\in K'' \cap E^{-1}(\ell^2)$ satisfies $\Ev(\qq)=(q,v)$.
Therefore, the open subset
\begin{align*}
U:=\big\{\qq=(q_0,q_1,...,q_{k-1})\in\Upsilon M\ \big|\ (q_0,\exp_{q_0}^{-1}(q_1))\neq(q,v)  \big\}
\end{align*}
contains the set of critical points $\crit(E_{\delta,k})\cap E^{-1}(\ell^2)$, and the classical Lusternik-Schnirelmann's theorem \cite[Theorem~1.1]{Viterbo:1997pi} implies that the cohomology class $\omega|_{U}\in H^{2n-1}(U)$ is non-zero. Consider the commutative diagram
\begin{equation*}
 \xymatrix{
     U \ar@{^{(}->}[rr]^{\incl}\ar[drr]_{\Ev|_U}  & & \Upsilon M
     \ar[d]^{\Ev} \\
  & & SM
 }  
\end{equation*}
Since $\omega$ is the generator of the image $\Ev^*(H^{2n-1}(SM))$,  $\omega|_{U}$ is the generator of the image $\Ev|_U^*(H^{2n-1}(SM))$. However, the homomorphism $\Ev|_U^*:H^{2n-1}(SM)\to H^{2n-1}(U)$ is zero, since the map $\Ev|_U$ is not surjective. This implies that $\omega|_U=0$ in $H^{2n-1}(U)$, which is a contradiction.

So far, we have proved that $g$ is Besse, and $\ell$ or $\ell-2\delta$ is a common multiple of the periods of the unit-speed geodesics.
Now, consider the critical sets
\begin{align*}
K & :=\crit(E)\cap\big(E^{-1}(\ell^2)\cup E^{-1}((\ell-2\delta)^2)\big),\\
K_{\delta,k} & :=\crit(E_{\delta,k})\cap E_{\delta,k}^{-1}(\ell^2)
\end{align*}
Let $\epsilon>0$ be small enough so that $(\ell^2,\ell^2+\epsilon)$ does not contain critical values of $E_{\delta,k}$. Since $c_g(\mu)=\ell$, the relative cohomology group
\begin{align*}
H^d(\Upsilon M^{<\ell^2+\epsilon},\Upsilon M^{<\ell^2})\cong H^{d-\ind(E_{\delta,k},K_{\delta,k})}(K_{\delta,k}) 
\end{align*}
is nontrivial. In particular
\begin{align}
\label{e:ind_K_delta_k}
\ind(E_{\delta,k},K_{\delta,k})\leq d.
\end{align}
Since $\delta<\rho$, $2\delta$ is smaller than the minimal period of the unit-speed geodesics of $(M,g)$. This readily implies that the values $\ell$ and $\ell-\delta$ cannot both be common periods for the unit-speed geodesics of $(M,g)$, and we have two possible cases:
\begin{itemize}
\item If $\ell$ is a common period for the unit-speed geodesics of $(M,g)$, then $K_{\delta,k}$ does not contain zig-zag closed geodesics, and indeed $K=K_{\delta,k}$.

\item If $\ell-2\delta$ is a common period for the unit-speed geodesics of $(M,g)$, then $K_{\delta,k}$ contains only zig-zag closed geodesics, and more precisely of those closed geodesics contained in $K$.
\end{itemize}
In both cases, Lemmas~\ref{l:indices} and~\ref{l:indices_zigzag}, together with the inequality~\eqref{e:ind_K_delta_k}, imply
\[
\ind(E,K)
\leq
\ind(E_{\delta,k},K_{\delta,k})
\leq
d.
\qedhere
\] 
\end{proof}

\subsection{Two subordinated homology classes in the Zoll case}

In this subsection, we will consider a Zoll Riemannian manifold, and prove the ``if'' claim in Theorem~\ref{t:main}.

\begin{lem}\label{l:Zoll}
Let $M$ be a closed manifold of dimension $n\geq 2$ admitting a simple Zoll Riemannian metric, and $g$ a Zoll Riemannian metric on $M$ whose unit-speed closed geodesics have minimal period $\ell>0$. For each $\delta\in(0,\rho)$ and for each integer $k>\overline k(\ell,\delta)$, we consider the space $\Upsilon M=\Upsilon_{\delta,k}M$. For each $\epsilon\in(4\delta^2,\ell^2)$,
there exists a relative homology class
\[h\in H_{i(M)+2n-1}(\Upsilon M^{\leq \ell^2},\Upsilon M^{< \epsilon})\]
such that $h$ and $h\frown \omega|_{\Upsilon M^{\leq\ell^2}}$ are not in the kernel of the homomorphism
\begin{align*} 
H_{*}(\Upsilon M^{\leq \ell^2},\Upsilon M^{<\epsilon})
\ttoup^{\incl_*}
H_{*}(\Lambda M,\Lambda M^{<\epsilon}).
\end{align*}
\end{lem}

\begin{proof}
Let $K:=\crit(E)\cap E^{-1}(\ell^2)$ be the critical manifold of the non-iterated closed geodesics. By Lemma~\ref{l:indices}, for each $\delta\in(0,\rho)$ and integer $k>\overline k(\ell,\delta)$, we have $K\subset\Upsilon M:=\Upsilon_{\delta,k}M$ and, for each $\gamma_{\qq}\in K$, 
\begin{equation}
\label{e:same_indices}
\begin{split}
\ind(E,\gamma_{\qq})&=\ind(E_{\delta,k},\qq)=i(M),
\\
\nul(E,\gamma_{\qq})&=\nul(E_{\delta,k},\qq)=2n-2.
\end{split}
\end{equation}
We denote by $G$ the Riemannian metric on $\Upsilon M$ induced by $g$, i.e.
\begin{align*}
G(\vv,\ww)=\sum_{i\in\Z_{k}} g(v_i,w_i),
\qquad
\forall\vv,\ww\in\Tan_{\qq}\Upsilon M.
\end{align*}
Let $\pi:N\to K$ be the negative bundle of $E_{\delta,k}$ at $K$. Namely, for each $\qq\in K$, the fiber $\pi^{-1}(\qq)\subset\Tan_{\qq}\Upsilon M$ is the negative eigenspace  of the symmetric linear map $H_{\qq}:\Tan_{\qq}\Upsilon M\to\Tan_{\qq}\Upsilon M$ defined by
$G( H_{\qq}\,\cdot,\cdot)= \diff^2E_{\delta,k}(\qq)$.
The rank of this vector bundle is $i(M)$, according to~\eqref{e:same_indices}.
For each $r>0$, we set $N_r\subset N$ to be the $r$-neighborhood of the 0-section, measured with respect to $G$. With a slight abuse of notation, we still denote by $\exp$ the exponential map of $(\Upsilon M,G)$. We choose $r>0$ to be small enough so that $\exp|_{N_r}$ is a well defined diffeomorphism onto a neighborhood of $K$ in $\Upsilon M$, and $E_{\delta,k}(\exp_{\qq}(\vv))<E_{\delta,k}(\qq)$ for all $(\qq,\vv)\in N_r$ with $\vv\neq0$. Since $E$ has no critical values in the interval $(\ell^2, (\ell+\delta)^2)$, the arrows in the following commutative diagram are isomorphisms
\begin{align*}
 \xymatrix{
     H_*(N_r,\partial N_r) \ar[rr]^{\exp_*\ \ \ \ \ \ \ \ \ }_{\cong\ \ \ \ \ \ \ \ \ }\ar[ddrr]_{\exp_*}^{\cong}  & & H_{*}(\Upsilon M^{\leq\ell^2},\Upsilon M^{<\ell^2})\ar[dd]^{\incl_*}_{\cong} \\\\
  & & H_{*}(\Lambda M^{<(\ell+\delta)^2},\Lambda M^{<\ell^2})
 }  
\end{align*}
see \cite[Theorem~D.2]{Goresky:2009fq}. Since $E$ is a perfect functional, the exponential map also induces an injective homomorphism
\begin{align*}
 \exp_*:H_*(N_r,\partial N_r)\hookrightarrow H_{*}(\Lambda M,\Lambda M^{<\ell^2})
\end{align*}
Since $\ell^2$ is the smallest positive critical value of $E$, the restriction $E_{\delta,k}$ has no critical values in the interval $(4\delta^2,\ell^2)$. For each $\epsilon\in(4\delta^2,\ell^2)$, if we denote by $\phi_t$ the anti-gradient flow of $E_{\delta,k}$, we can fix $t>0$ large enough so that 
\[\phi_t\circ\exp(\partial N_r)\subset\Upsilon M^{<\epsilon}.\] 
If we set $\iota:=\phi_t\circ\exp$, the induced homomorphisms $\iota_*$ and $\incl_*\circ \iota_*$ in the following commutative diagram must be injective
\begin{align*}
 \xymatrix{
    \Big. H_*(N_r,\partial N_r)\ \ar@{^{(}->}[rr]^{\exp_*}  \ar@{^{(}->}[dd]_{\iota_*}  & & H_{*}(\Lambda M,\Lambda M^{<\ell^2})  \\\\
H_{*}(\Upsilon M^{\leq\ell^2},\Upsilon M^{<\epsilon})\ \ar[rr]^{\incl_*}
  & &  H_{*}(\Lambda M,\Lambda M^{<\epsilon}) \ar[uu]_{\incl_*}^{\cong}
 }  
\end{align*}
Since the closed geodesics in $K$ are not iterated, the negative bundle $N\to K$ is orientable. If $\tau\in H^{i(M)}(N_r,\partial N_r)$ denotes its Thom class with respect to any orientation, we have a Thom isomorphism
\begin{align*}
H^*(N_r)\to H^{*+i(M)}(N_r,\partial N_r),
\qquad
\mu\mapsto \tau\smile \mu.
\end{align*}
If we denote by $\omega'$ the generator of $H^{2n-1}(N_r)\cong H^{2n-1}(SM)$, and by $h'$ the generator of $H_{i(M)+2n-1}(N_r,\partial N_r)$, then $h'\frown\omega'$ is the generator of $H_{i(M)}(N_r,\partial N_r)$. Consider the evaluation map $\Ev:\Upsilon M^{\leq\ell^2}\to SM$ of Equation~\eqref{e:evaluation}, which is injective in cohomology according to Lemma~\ref{l:injectivity}. If we denote by $0_N\subset N$ the 0-section of $N$, the composition $\Ev\circ \iota|_{0_N}:0_N\to SM$ is clearly a homeomorphism. Therefore, up to changing the sign of $\omega'$,
\[\omega'=\iota^*(\omega|_{\Upsilon M^{\leq\ell^2}}).\] 
We set 
\[h:=\iota_*h'\in H_{i(M)+2n-1}(\Upsilon M^{\leq\ell^2},\Upsilon M^{<\epsilon}),\] 
and notice that
\begin{align*}
h\frown\omega|_{\Upsilon M^{\leq\ell^2}}=(\iota_*h')\frown\omega|_{\Upsilon M^{\leq\ell^2}}=\iota_*(h' \frown \iota^*(\omega|_{\Upsilon M^{\leq\ell^2}}))= \iota_*(h'\frown\omega')\neq 0
\end{align*}
in $H_{i(M)}(\Upsilon M^{\leq\ell^2},\Upsilon M^{<\epsilon})$. 
\end{proof}

In the following lemma, we will employ the notation of the introduction, and consider the cohomology classes $\alpha$ and $\omega\smile j^*\alpha$ from Equations \eqref{e:omega} and \eqref{e:alpha}.

\begin{lem}
\label{l:Zoll_minimax}
Let $M$ be a closed manifold of dimension $n\geq 2$ admitting a simple Zoll Riemannian metric, and $g$ a Zoll Riemannian metric on $M$ whose unit-speed closed geodesics have minimal period $\ell>0$. For each $\delta\in(0,\rho)$ and integer $k>\overline k(\ell,\rho/\sqrt2)$, consider the space $\Upsilon M=\Upsilon_{\delta,k}M$. Then  $c_g(j^*\alpha)=c_g(\omega\smile j^*\alpha)=\ell^2$.
\end{lem}

\begin{proof}
By Lemma~\ref{l:Zoll}, there exists a homology class \[h\in H_{i(M)+2n-1}(\Upsilon M^{\leq \ell^2},\Upsilon M^{< \epsilon})\]
such that both $h$ and $h\frown \omega|_{\Upsilon M^{\leq\ell^2}}$ are mapped to non-zero homology classes under the homomorphism
\begin{align*} 
(j_{\ell^2})_*=\incl_*:
H_{*}(\Upsilon M^{\leq \ell^2},\Upsilon M^{<\epsilon})
\to
H_{*}(\Lambda M,\Lambda M^{<\epsilon}).
\end{align*}
Equation~\eqref{e:cohomology_loop_space} implies that
\begin{align*}
H^{i(M)}(\Lambda M,\Lambda M^{<4\rho^2})&\cong H^{i(M)}(\Lambda M,M)\cong\Z,\\
H^{i(M)-1}(\Lambda M,\Lambda M^{<4\rho^2})&\cong H^{i(M)-1}(\Lambda M,M)=0.
\end{align*}
Therefore, by the universal coefficient theorem,
\begin{align*}
H^{i(M)}(\Lambda M,\Lambda M^{<4\rho^2})\cong\mathrm{Hom}\big(H_{i(M)}(\Lambda M,\Lambda M^{<4\rho^2}),\Z\big),
\end{align*}
and the generator $\alpha\in H^{i(M)}(\Lambda M,\Lambda M^{<4\rho^2})$ must satisfy 
\[
(\omega \smile j_{\ell^2}^*\alpha)h
=
(j_{\ell^2}^*\alpha)(h\frown\omega)
=
\alpha((j_{\ell^2})_*(h\frown\omega))
\neq0.
\]
This implies that $c_g(j^*\alpha)\leq c_g(\omega\smile j^*\alpha)\leq\ell^2$. On the other hand, $\ell$ is the smallest critical value of the energy $E|_{\Upsilon M}$ above the global minimum $4\rho^2$, and therefore we have the opposite inequality $c_g(j^*\alpha)\geq\ell^2$.
\end{proof}

\subsection{Two subordinated homology classes for arbitrary metrics}

Let $M$ be a closed Riemannian manifold of dimension $n\geq2$ equipped with a Zoll Riemannian metric $g_0$ and with an arbitrary Riemannian metric $g_1$. Their convex combinations 
\[
g_s:=(1-s)g_0+sg_1, \qquad s\in[0,1],\] 
give a path of Riemannian metrics. We will denote with a subscript or superscript $s$ the usual Riemannian objects associated with the Riemannian metric $g_s$: the exponential map $\exp^{(s)}:\Tan M\to M$, the Riemannian distance $d_s:M\times M\to[0,\infty)$, the injectivity radius $\rho_s=\injrad(M,g_s)$, and the energy  $E_s:\Lambda M\to[0,\infty)$. We set 
\begin{align*}
 d_{\max}(q_0,q_1) & :=\max_{s\in[0,1]} d_s(q_0,q_1),\quad \forall q_0,q_1\in M,\\
 \rho_{\min} & :=\min\{\rho_s\ |\ s\in[0,1]\}>0,\\
 c & :=\min\big\{\|v\|_{g_0}\|v\|_{g_s}^{-1}\ \big|\ v\in \Tan M\setminus\zerosection,\ s\in[0,1]\big\}\in(0,1],\\
 \delta_{\max} & :=\frac{c\,\rho_{\min}}{2}.
\end{align*}
We fix $\delta_0\in(0,\delta_{\max})$ small enough, $\epsilon_0:=8\delta_0^2$, and $\epsilon_1:=4\rho_1^2$ so that we have the inclusion of sublevel sets 
\[\{E_0<\epsilon_0\}\subseteq\{E_1<\epsilon_1\}\subset\Lambda M.\]
Since both these sublevel sets can be deformed onto the space of constant loops $M\subset \Lambda M$, the inclusion induces a homology isomorphism
\begin{align*}
H_*(\Lambda M,\{E_0<\epsilon_0\})
\ttoup^{\incl_*}_{\cong}
H_*(\Lambda M,\{E_1<\epsilon_1\}).
\end{align*}

We denote by $\ell_0>0$ the minimal period of the unit-speed geodesics of the Zoll metric $g_0$. By Lemma~\ref{l:Zoll}, for each integer $k_0>\overline k_0(\ell_0,\delta_0)$, if we set 
\[\Upsilon^{(0)} M=\Upsilon_{\delta_0,k_0}^{(0)}M,\] 
there exists a relative homology class
$h\in H_{i(M)+2n-1}(\Upsilon^{(0)} M,\Upsilon^{(0)} M^{< \epsilon_0})$
such that $h$ and $h\frown \omega_{0}$ are not in the kernel of the homomorphism
\begin{align} 
\label{e:inclusion_last}
H_{*}(\Upsilon^{(0)} M,\Upsilon^{(0)} M^{<\epsilon_0})
\ttoup^{\incl_*}
H_{*}(\Lambda M,\{E_0<\epsilon_0\}).
\end{align}
Here, $\omega_{0}\in H^{2n-1}(\Upsilon^{(0)}M)$ is the cohomology class~\eqref{e:omega} for the Riemannian metric $g_0$.  Let $\sigma$ be a relative cycle representing $h$, which we can see as a continuous map of the form 
\[\sigma:(\Sigma,\partial \Sigma)\to(\Upsilon^{(0)} M,\Upsilon^{(0)} M^{<\epsilon_0})\subset (\Lambda M,\{E_0<\epsilon_0\})\] 
for a suitable simplicial complex $\Sigma$ with simplicial boundary $\partial \Sigma$. Hereafter, we will treat the points $\sigma(z)$ as elements of the loop space $\Lambda M$.

\begin{lem}
\label{l:T}
For each $\delta_1>0$ small enough, there exists a continuous function $T:\Sigma\to(0,1)$ such that
\begin{align*}
d_{1}(\sigma(z)(0),\sigma(z)(T(z)))=\delta_1,
\qquad\forall z\in\Sigma.
\end{align*}
\end{lem}

\begin{proof}
We will denote by  
$\tau_1:\Upsilon^{(0)} M\to(0,1)$
the function \eqref{e:tau_1} associated to $g_0$, and by $SM$ the unit tangent bundle of $(M,g_0)$. If $\tau''\in(0,\delta_0)$ is sufficiently small, the function
\begin{gather*}
F:SM\times[0,\tau'')\to [0,\infty),
\\
F(q,v,t)
=
d_1(q,\exp_q^{(0)}(tv))^2
=
\|(\exp_q^{(1)})^{-1}\circ\exp_q^{(0)}(tv)\|_{g_1}^2
\end{gather*}
is smooth. For each $(q,v)\in SM$, the function $F(q,v,\cdot)$ has a unique global minimizer at $t=0$, and $F(q,v,0)=0$. Since 
\begin{align*}
(\exp_x^{(1)})^{-1}\circ\exp_x^{(0)}(tv)=tv+o(t),
\end{align*}
we readily see that there exists $\tau'\in(0,\tau'')$ such that
\begin{align*}
\tfrac{\diff}{\diff t} F(q,v,t)>0,
\qquad\forall (q,v)\in SM,\ t\in(0,\tau').
\end{align*}
By the implicit function theorem, for each $\delta_1>0$ small enough there exists a smooth function $T':SM\to(0,\tau')$ such that
\begin{align*}
F(q,v,T'(q,v))=\delta_1,\qquad
\forall (q,v)\in SM.
\end{align*}
Now, for each $z\in\Sigma$, the curve $\gamma_z:=\sigma(z)|_{[0,\tau_1(\sigma(z))]}$ is a geodesic of $(M,g_0)$ with speed 
$\|\dot\gamma(0^+)\|_{g_0}=\delta_0/\tau_1(\sigma(z))$. If we set $q_z:=\gamma_z(0)$ and $v_z:=\dot\gamma_z(0^+)/\|\dot\gamma_z(0^+)\|_{g_0}$, we have
\[
F(q_z,v_z,t)=d_1\big(\gamma_z(0),\gamma_z(t\,\tau_1(\sigma(z))/\delta_0)\big)^2.
\]
The desired continuous function is given by 
$T(z):=T'(q_z,v_z)\,\tau_1(\sigma(z))/\delta_0$.
\end{proof}

For each $\tau\in(0,1)$, we introduce the subspace 
\begin{align*}
U_\tau:=\left\{\gamma\in\Lambda M
\ \left|\ 
  \begin{array}{@{}l@{}}
    \gamma(0)\neq\gamma(t)\ \ \forall t\in[0,\tau] \vspace{5pt}\\ 
   \displaystyle
   \max_{t\in(0,\tau]}d_{\max}(\gamma(0),\gamma(t))<\rho_{\min}
  \end{array}
\right.\right\}.
\end{align*}
We fix $\tau\in(0,1)$ small enough to that the support of our cycle $\sigma(\Sigma)$ is contained in $U_\tau$. 

\begin{lem}
\label{l:sigma_s}
For each $\delta_1\in(0,\rho_{\min})$ small enough, integer $k_1\in\N$ large enough, and $\tau>0$ small enough, if we set 
\[\Upsilon^{(1)}M:=\Upsilon_{\delta_1,k_1}^{(1)}M,\]  
there exists a homotopy $\sigma_s:\Sigma\to U_\tau$, with $s\in[0,1]$, such that $\sigma_0=\sigma$, $\sigma_1(\Sigma)\subset U_\tau\cap\Upsilon^{(1)}M$, and $s\mapsto E_1(\sigma_s(z))$ is monotonically decreasing for all $z\in\Sigma$.
\end{lem}

\begin{proof}
We fix a small enough $\delta_1\in(0,\rho_{\min})$ so that Lemma~\ref{l:T} holds with an associated function $T:\Sigma\to(0,1)$. We also fix $\tau\in(0,\min T)$, and a large enough $k_1\in\N$ so that, if we set \[\Upsilon^{(1)}M:=\Upsilon^{(1)}_{\delta_1,k_1}M,\] we have
\begin{align*}
\max E_1\circ\sigma < \sup_{\Upsilon^{(1)}M} E_1,
\end{align*}
see~\eqref{e:max_E_delta_k}. We fix $z\in\Sigma$ and $\gamma_0:=\sigma(z)$, and define its deformation $\gamma_s=\sigma_s(z)\in U_\tau$, for $s\in[0,1]$, as follows. We set
\begin{align*}
t_i=t_i(z):=T(z) + \frac{1-T(z)}{k-1}(i-1),
\qquad
i=1,...,k,
\end{align*}
so that $0=:t_0<t_1<...<t_k=1$. The first half of the deformation, for $s\in[0,1/2]$, is the usual Morse shortening process: we set $r_{i,s}:=(1-2s)t_i+2st_{i+1}$ for $i=0,...,k-1$; we define $\gamma_s|_{[t_i,r_{i,s}]}$ the be the shortest $g_1$-geodesic such that $\gamma_s(t_i)=\gamma_0(t_i)$ and $\gamma_s(r_{i,s})=\gamma_0(r_{i,s})$, and we set $\gamma_s|_{[r_{i,s},t_{i+1}]}=\gamma_0|_{[r_{i,s},t_{i+1}]}$. 

The curve $\gamma_{1/2}$ is a broken geodesic whose first portion $\gamma_{1/2}|_{[t_0,t_1]}$ has $g_1$-length $\delta_1$. The second half of the deformation, for $s\in[1/2,1]$, is just a time reparametrization of $\gamma_{1/2}$ that will make it belong to $\Upsilon^{(1)}M$. We set 
\[
\qq=(q_0,...,q_{k_1-1}):=(\gamma_{1/2}(t_0),\gamma_{1/2}(t_1),...,\gamma_{1/2}(t_{k-1})).
\]
We will denote by  
$\tau_i:=\tau_i(\sigma(z))$
the times values \eqref{e:tau_0}, \eqref{e:tau_1}, and \eqref{e:tau_i} associated to $g_1$. For each $s\in[1/2,1]$, $i=0,...,k-1$, and $r\in[0,1]$,  we set
\begin{align*}
\gamma_s\big( (2s-1)((1-r)\tau_i + r\tau_{i+1}) + (2-2s)((1-r)t_i + r t_{i+1}) \big)\qquad &\\
:=
\gamma_{1/2}\big((1-r)t_i + r t_{i+1} \big).&
\qedhere
\end{align*}
\end{proof}

For each $s\in[0,1]$ and $t\in(0,\tau]$, we have the evaluation map
\begin{align*}
\Ev_{s,t}:U_\tau\to \Tan M\setminus \zerosection,
\qquad
\Ev_{s,t}(\gamma)=(\exp_{\gamma(0)}^{(s)})^{-1}(\gamma(t)).
\end{align*}
Since $\Ev_{s,t}$ depends continuously on the pair $(s,t)$, the cohomology homomorphism $\Ev_{s,t}^*:H^{2n-1}(\Tan M\setminus \zerosection)\to H^{2n-1}(U_\tau)$ is actually independent of $(s,t)$. We denote a generator of its image by
\[
\Omega_\tau
\in
\Ev_{s,t}^*(H^{2n-1}(\Tan M\setminus \zerosection))
\subset
H^{2n-1}(U_\tau).
\]

\begin{lem}
\label{l:Omega}
Up to changing the sign of $\omega_0$ and $\omega_1$, we have 
\[\Omega_\tau|_{U_\tau\cap\Upsilon^{(s)}M}=\omega_s|_{U_\tau\cap\Upsilon^{(s)}M},
\qquad\forall s\in\{0,1\}.\]
\end{lem}

\begin{proof}
We fix $s\in\{0,1\}$, and denote by  
$\tau_1:\Upsilon^{(s)} M\to(0,1)$
the function \eqref{e:tau_1} and by $\Ev:\Upsilon^{(s)}M\to \Tan M\setminus\zerosection$ the evaluation map~\eqref{e:evaluation} associated to $g_s$. Notice that the codomain of $\Ev$ in~\eqref{e:evaluation} is the unit tangent bundle $SM$ of $(M,\delta_s^{-2}g_s)$. However, since the the inclusion $SM\hookrightarrow\Tan M\setminus\zerosection$ is a homotopy equivalence, the cohomology class $\omega_s$ will also be the generator of $\Ev^*(H^*(\Tan M\setminus\zerosection))$. Since
\begin{align*}
\Ev(\gamma_{\qq})=\Ev_{s,\tau_1(\qq)}(\gamma_{\qq}),
\qquad\forall\gamma_{\qq}\in U_\tau\cap\Upsilon^{(s)}M,
\end{align*}
we readily see that $\Ev|_{U_\tau\cap\Upsilon^{(s)}M}$ and $\Ev_{s,t}|_{U_\tau\cap\Upsilon^{(s)}M}$, for all $t\in(0,\tau]$, are homotopic maps. Therefore, both the restrictions $\Omega_\tau|_{U_\tau\cap\Upsilon^{(s)}M}$ and $\omega_s|_{U_\tau\cap\Upsilon^{(s)}M}$ are generators of the image $\Ev|_{U_\tau\cap\Upsilon^{(s)}M}^*(\Tan M\setminus\zerosection)$, and, up to changing the sign of $\omega_s$, they coincide.
\end{proof}

Since $h$ and $h\frown\omega_0$ are mapped to non-zero classes in $H_*(\Lambda M,\{E_0<\epsilon_0\})$ by the homomorphism~\eqref{e:inclusion_last}, Lemmas~\ref{l:sigma_s} and~\ref{l:Omega} imply that, for each $s\in\{0,1\}$, $[\sigma_s]$ and $[\sigma_s]\cap\omega_s$ are non-trivial relative homology classes in 
\[
H_*(U_\tau\cap\Upsilon^{(s)}M,U_\tau\cap\Upsilon^{(s)}M^{<\epsilon_s}).
\] 
If we consider the homomorphism induced by the inclusion
\begin{align*}
\iota^{(s)}_*:H_*(U_\tau\cap\Upsilon^{(s)}M,U_\tau\cap\Upsilon^{(s)}M^{<\epsilon_s})
\ttoup^{\incl_*}
H_*(U_\tau,U_\tau\cap\{E_1<\epsilon_1\}),
\end{align*}
we have
\begin{align*}
\iota^{(0)}_*[\sigma_0]=\iota^{(1)}_*[\sigma_1]=[\sigma],
\qquad
\iota^{(0)}_*([\sigma_0]\frown\omega_0)=\iota^{(1)}_*([\sigma_1]\frown\omega_1)=[\sigma]\frown\Omega_\tau.
\end{align*}
In particular, if we see $\sigma_s$ and $\sigma_s\frown\omega_s$ as relative cycles in $(\Lambda M,\{E_1<\epsilon_1\})$, we have $[\sigma_0]=[\sigma_1]$ and $[\sigma_0\frown\omega_0]=[\sigma_1\frown\omega_1]$ in $H_*(\Lambda M,\{E_1<\epsilon_1\})$.

Now, consider the generator
$\alpha\in H^{i(M)}(\Lambda M,\{E_1<\epsilon_1\})\cong\Z$, and the inclusion 
\[j_1:(\Upsilon^{(1)} M,\Upsilon^{(1)} M^{<\epsilon_1})\hookrightarrow(\Lambda M,\{E_1<\epsilon_1\}).\]

\begin{lem}
\label{l:cohomology}
For each $\delta_1\in(0,\rho_1)$ small enough and $k_1\in\N$ large enough, if we set 
\[\Upsilon^{(1)}M:=\Upsilon_{\delta_1,k_1}^{(1)}M,\]
we have $\omega_1\smile j_1^*\alpha\neq0$ in $H^{i(M)}(\Upsilon^{(1)} M,\Upsilon^{(1)} M^{<\epsilon_1})$.
\end{lem}

\begin{proof}
Let $\delta_1\in(0,\rho_1)$ be small enough and $k_1\in\N$ large enough so that Lemma~\ref{l:sigma_s} holds. Since $H^{i(M)-1}(\Lambda M,\{E_1<\epsilon_1\})$ is trivial, the universal coefficient theorem implies that
\begin{align*}
H^{i(M)}(\Lambda M,\{E_1<\epsilon_1\})\cong\mathrm{Hom}\big( H_{i(M)}(\Lambda M,\{E_1<\epsilon_1\}),\Z \big).
\end{align*}
This, together with the facts that $\alpha$ is the generator of $H^{i(M)}(\Lambda M,\{E_1<\epsilon_1\})$ and that $(j_0)_*([\sigma_0]\frown\omega_0)$ is non-zero in $H_{i(M)}(\Lambda M,\{E_1<\epsilon_1\})$, implies that
\begin{align*}
\alpha((j_0)_*([\sigma_0]\frown\omega_0))
\neq0.
\end{align*}
Therefore, we conclude
\begin{align*}
(\omega_1\smile j_1^*\alpha)[\sigma_1]
&=
(j_1^*\alpha)([\sigma_1]\frown\omega_1)
=
\alpha((j_1)_*([\sigma_1]\frown\omega_1))\\
&=
\alpha((j_0)_*([\sigma_0]\frown\omega_0))
\neq0.
\qedhere
\end{align*}
\end{proof}

\begin{proof}[Proof of Theorem~\ref{t:main}]
Let $M$ be a closed manifold of dimension $n\geq2$ admitting a simple Zoll Riemannian metric, and $g$ a Riemannian metric on $M$.
Lemma~\ref{l:cohomology} implies that $\omega\smile j^*\alpha\neq0$ in the relative homology group $H^{i(M)}(\Upsilon M,\Upsilon M^{<4\rho^2})$. 

If $g$ is a Zoll Riemannian metric whose unit-speed geodesics have minimal period $\ell$, then Lemma~\ref{l:Zoll_minimax} implies that $c_g(j^*\alpha)=c_g(\omega\smile j^*\alpha)=\ell^2$. Conversely, assume that $c_g(j^*\alpha)=c_g(\omega\smile j^*\alpha)=:\ell^2$.
We can apply Lemma~\ref{l:main_lemma} with $d=i(M)$ and $\mu=j^*\alpha$, and infer that $g$ is a Besse Riemannian metric, and either $\ell$ or $\ell-2\delta$ is a common multiple of the periods of the unit-speed geodesics of $(M,g)$. Moreover, the critical set 
\[K:=\crit(E)\cap\big(E^{-1}(\ell^2)\cup E^{-1}((\ell-\delta)^2)\big)\cong SM,\]
has Morse index $\ind(E,K)\leq i(M)$. Since $i(M)$ is the minimal Morse index of a closed geodesic, we have
\begin{align}
\label{e:ind_K}
\ind(E,K) = i(M).
\end{align}
Now, let us further require $M$ to be simply connected and spin, and assume by contradiction that $g$ is not a Zoll Riemannian metric. We are now going to employ two results due to Radeschi and Wilking. Since $(M,g)$ is a simply connected Besse manifold, by \cite[Theorem~D]{Radeschi:2017dz} the energy functional $E:\Lambda M\to[0,\infty)$ is perfect for the $S^1$-equivariant singular cohomology with rational coefficients $H^*_{S^1}(\cdot\,;\Q)$. Moreover, since $(M,g)$ is an orientable and spin Besse manifold, by \cite[Corollary~C]{Radeschi:2017dz} the negative bundles of all the critical manifolds of $E$ are orientable. This, in turn, implies that all critical manifolds of $E$ are homologically visible, and, if we set 
\[
K_{i(M)}
:=
\big\{ \gamma\in\crit(E)\ \big|\ \ind(E)=i(M) \big\},
\]
we have
\begin{align*}
H^{i(M)}_{S^1}(\Lambda M,M;\Q)\cong H^{0}_{S^1}(K_{i(M)};\Q).
\end{align*}
Namely, the rank of $H^{i(M)}_{S^1}(\Lambda M,M;\Q)$ is  the number of path-connected components of $K_{i(M)}$.
Clearly, $K$ is a path-connected component of $K_{i(M)}$. Since $g$ is Besse but not Zoll, Equation~\eqref{e:ind_K} implies that $K_{i(M)}\setminus K$ is not empty, and therefore 
\begin{align}
\label{e:rank_geq_2}
\rank\big(H^{i(M)}_{S^1}(\Lambda M,M;\Q)\big)\geq2. 
\end{align}
On the other hand, if we repeat the whole argument with a Zoll Riemannian metric $g_0$ instead of $g$, the critical set $K_{i(M)}$ becomes diffeomorphic to the unit tangent bundle $SM$, which is path-connected. This implies that 
\[\rank\big(H^{i(M)}_{S^1}(\Lambda M,M;\Q)\big)=1,\] and contradicts~\eqref{e:rank_geq_2}.
\end{proof}

\bibliography{_biblio}

\providecommand{\bysame}{\leavevmode\hbox to3em{\hrulefill}\thinspace}
\providecommand{\MR}{\relax\ifhmode\unskip\space\fi MR }
\providecommand{\MRhref}[2]{%
  \href{http://www.ams.org/mathscinet-getitem?mr=#1}{#2}
}
\providecommand{\href}[2]{#2}
\begin{thebibliography}{Wad75}

\bibitem[Bes78]{Besse:1978pr}
A.~L. Besse, \emph{Manifolds all of whose geodesics are closed}, Ergebnisse der
  Mathematik und ihrer Grenzgebiete, vol.~93, Springer-Verlag, Berlin-New York,
  1978.

\bibitem[Bot54]{Bott:1954aa}
R.~Bott, \emph{On manifolds all of whose geodesics are closed.}, Ann. of Math.
  \textbf{60} (1954), no.~2, 375--382.

\bibitem[Bot56]{Bott:1956sp}
\bysame, \emph{On the iteration of closed geodesics and the {S}turm
  intersection theory}, Comm. Pure Appl. Math. \textbf{9} (1956), 171--206.

\bibitem[BTZ83]{Ballmann:1983fv}
W.~Ballmann, G.~Thorbergsson, and W.~Ziller, \emph{Existence of closed
  geodesics on positively curved manifolds}, J. Differential Geom. \textbf{18}
  (1983), no.~2, 221--252.

\bibitem[GG81]{Gromoll:1981kl}
D.~Gromoll and K.~Grove, \emph{On metrics on {$S^2$} all of whose geodesics are
  closed}, Invent. Math. \textbf{65} (1981), 175--177.

\bibitem[GH09]{Goresky:2009fq}
M.~Goresky and N.~Hingston, \emph{Loop products and closed geodesics}, Duke
  Math. J. \textbf{150} (2009), no.~1, 117--209.

\bibitem[Lju66]{Ljusternik:1966tk}
L.~A. Ljusternik, \emph{The topology of the calculus of variations in the
  large}, Translated from the Russian by J. M. Danskin. Translations of
  Mathematical Monographs, Vol. 16, American Mathematical Society, Providence,
  R.I., 1966.

\bibitem[Mil63]{Milnor:1963rf}
J.~Milnor, \emph{Morse theory}, Based on lecture notes by M. Spivak and R.
  Wells. Annals of Mathematics Studies, No. 51, Princeton University Press,
  Princeton, N.J., 1963.

\bibitem[MS17]{Mazzucchelli:2017aa}
M.~Mazzucchelli and S.~Suhr, \emph{A characterization of {Z}oll {R}iemannian
  metrics on the 2-sphere}, arXiv:1711.11285, to appear in Bull. Lond. Math.
  Soc., 2017.

\bibitem[RW17]{Radeschi:2017dz}
M.~Radeschi and B.~Wilking, \emph{On the {B}erger conjecture for manifolds all
  of whose geodesics are closed}, Invent. Math. \textbf{210} (2017), 911--962.

\bibitem[Sam63]{Samelson:1963aa}
H.~Samelson, \emph{On manifolds with many closed geodesics}, Portugal. Math.
  \textbf{22} (1963), 193--196.

\bibitem[Vit97]{Viterbo:1997pi}
C.~Viterbo, \emph{Some remarks on {M}assey products, tied cohomology classes,
  and the {L}usternik-{S}hnirelman category}, Duke Math. J. \textbf{86} (1997),
  no.~3, 547--564.

\bibitem[Wad75]{Wadsley:1975sp}
A.~W. Wadsley, \emph{Geodesic foliations by circles}, J. Differ. Geom.
  \textbf{10} (1975), no.~4, 541--549.

\bibitem[Wil01]{Wilking:2001pi}
B.~Wilking, \emph{Index parity of closed geodesics and rigidity of {H}opf
  fibrations}, Invent. Math. \textbf{144} (2001), 281--295.

\bibitem[Zil77]{Ziller:1977rp}
W.~Ziller, \emph{The free loop space of globally symmetric spaces}, Invent.
  Math. \textbf{41} (1977), no.~1, 1--22.

\end{thebibliography}
\bibliographystyle{amsalpha}

\end{document}